\documentclass[11pt,a4paper]{article}
\usepackage{amsmath,amssymb,calc}
\usepackage[english]{babel}
\usepackage{a4wide}
\usepackage{theorem}
\usepackage{graphicx}
\usepackage[numbers]{natbib}
\usepackage[T1]{fontenc}
\usepackage{tikz}
\usepackage{xcolor}
\usepackage{pgfplots}
\usepackage{subcaption}
\usepackage{color} 
\usepackage{enumerate}
\usepackage[charter]{mathdesign}
\usepackage{mathrsfs}
\usepackage{tikz}
\usepackage{relsize}
\usepackage{bbm}

\newtheorem{theorem}{Theorem}[section]

\newtheorem{proposition}[theorem]{Proposition}

\numberwithin{equation}{section}

\makeatletter
\DeclareRobustCommand{\qed}{%
  \ifmmode 
  \else \leavevmode\unskip\penalty9999 \hbox{}\nobreak\hfill
  \fi
  \quad\hbox{\qedsymbol}}
\newcommand{\openbox}{\leavevmode
  \hbox to.77778em{%
  \hfil\vrule
  \vbox to.675em{\hrule width.6em\vfil\hrule}%
  \vrule\hfil}}
\newcommand{\qedsymbol}{\openbox}
\newenvironment{proof}[1][\proofname]{\par
  \normalfont
  \topsep6\p@\@plus6\p@ \trivlist
  \item[\hskip\labelsep\itshape
    #1.]\ignorespaces
}{%
  \qed\endtrivlist
}
\newcommand{\proofname}{Proof}
\makeatother

\newcommand{\E}{\mathbb{E}}
\renewcommand{\P}{\mathbb{P}}

\providecommand{\href}[2]{#2}

\newcommand{\e}{{\rm e}}

\newcommand{\dd}[1]{{\rm d}#1}

\newcommand{\Rightarrowd}{\,{\buildrel d \over \Rightarrow}\,}
\newcommand{\equalD}{\,{\buildrel d \over =}\,}

\def \F {\Phi}
\def \f {\varphi}
\def \Pois {{\rm Pois}}
\def \l {\lambda}
\def \sl {{s_\lambda}}
\def \ee {{\rm e}}
\def \eps {\varepsilon}
\def \b {\beta}

\def \l {\lambda}

\def \g {\gamma}

\def\P{{\mathbb{P}}}

\usepackage[toc]{multitoc}

\definecolor{col1}{rgb}{0.8, 0.01568, 0.}
\definecolor{col2}{rgb}{0.9372549019607843, 0.62745, 0.1}
\definecolor{col3}{rgb}{0.9686274509803922, 0.87, 0.}
\definecolor{col4}{rgb}{0.3176470588235294, 0.58, 0.0784313}
\definecolor{col5}{rgb}{0.22745098, 0.2392156862,0.43}
\definecolor{col6}{rgb}{0.5019607843137255, 0.16, 0.4470}

\newcommand{\chp}[1]{}

\title{Economies-of-scale in many-server {queueing} systems:
tutorial and partial review of the QED Halfin-Whitt heavy-traffic regime}
\author{Johan S.H. van Leeuwaarden
\and
Britt W.J. Mathijsen
\and
Bert Zwart
}

\begin{document}
\maketitle

\begin{abstract}
Multi-server {queueing}  systems describe situations in which users require service from multiple parallel servers. Examples include check-in lines at airports, waiting rooms in hospitals, queues in contact centers, data buffers in wireless networks, and delayed service in cloud data centers. These are all situations with jobs (clients, patients, tasks) and servers (agents, beds, processors) that have large capacity levels, ranging from the order of tens (checkouts) to thousands (processors).
This survey investigates how to design such systems to exploit resource pooling and economies-of-scale.
In particular, we review the mathematics behind the Quality-and-Efficiency Driven (QED) regime, which lets the system operate close to full utilization, while the number of servers grows simultaneously large and delays remain manageable. 
{
Aimed at a broad audience, we describe in detail the mathematical concepts for the basic Markovian many-server system, and only provide sketches or references for more advanced settings related to e.g. load balancing, overdispersion, parameter uncertainty, general service
requirements and queueing networks. 
While serving as a partial survey of a massive body of work, the tutorial is not aimed to be exhaustive. }

\end{abstract}

\tableofcontents
\section{Introduction}
Multi-server systems describe situations in which users require service from multiple parallel servers.
Classical examples of such systems include call centers \cite{Erlang1917,Palm1957,Whitt1999,Gans2003,Borst2004,Brown2005,Zeltyn2005,Bassamboo2009,Khudyakov2006}, health care delivery \cite{Armony2015,Green2007,YomTov2010,Gupta2007}, and communication systems \cite{Anick1982,Kelly1985,johanthesis,Tan2012}.
In all settings, one can think of such systems as being composed of \textit{jobs} and \textit{servers}.
In call centers, jobs are customers' requests for help from one of the agents (servers).
In communication networks, the data packets are the jobs and the communication channels are the servers.
The system scale may refer to the size of the client base it caters to, or the magnitude of its capacity, or both.

Next to the central notions of jobs and servers, most multi-server systems are subject to uncertainty and hence give rise to stochastic systems.
Although arrival volumes over a certain planning horizon can be anticipated to some extent, for instance through historical data and forecasting methods, it is challenging to predict with certainty future arrival patterns.
Moreover, job sizes are typically random as well, adding more uncertainty.
This intrinsic stochastic variability is a predominant cause of delay experienced by jobs in the system, which is why stochastic models have proved instrumental in both quantifying and improving the operational performance of multi-server systems. Queueing theory provides the mathematical tools to analyze such stochastic models, and to evaluate and improve system performance. Queueing theory can also serve to reveal capacity-sizing rules that prescribe how to scale multi-server systems, in terms of matching capacity with demand, to meet certain performance targets. Often a trade-off exists between high system utilization and short delays.

\vspace{.2cm}
\noindent
\textbf{Effects of resource pooling.}
Let us first demonstrate the effects of resource pooling for the most basic multi-server queueing model, the $M/M/s$ queue. This model assumes that jobs arrive according to a Poisson process, that their service times form an i.i.d.~sequence of exponential random variables, and that jobs are processed in order of arrival by one of the $s$ parallel servers. Delayed jobs are temporarily stored in an infinite-sized buffer. The three parameters that characterize this model are: the arrival rate $\lambda$, the mean processing time $1/\mu$  and the number of servers $s$. We denote the number of jobs in the system at time $t$ by $Q(t)$. The process $(Q(t))_{t \geq 0}$ is a continuous-time Markov chain with state space $\{0, 1, 2,\ldots\}$. The birth rate $\lambda$ is constant and the death rate is $\mu\cdot\min\{k,s\}$ when there are $k$ jobs in the system.
 Observe now that we can change the time scale by considering the process $(Q(t\mu ))_{t \geq 0}$, so that a busy server completes one job per unit of time. This allows us to consider the case $\mu=1$ without loss of
 generality.

To illustrate the operational benefits of sharing resources, we compare a system of $s$ separate $M/M/1$ queues, each serving a Poisson arrival stream with rate $\lambda<1$, against one $M/M/s$ queue with arrival rate $\lambda s$.
The two systems thus face the same workload $\lambda$ per server.
We now fix the value of $\lambda$ and vary $s$.
Obviously, the delay and queue length distribution in the first scenario with parallel servers are unaffected by the parameter $s$, since there is no interaction between the single-server queues.
This lack of coordination tolerates an event of having an idle server, while the total number of jobs in the system exceeds $s$, therefore wasting resource capacity. Such an event cannot happen in the many-server scenario, due to the central queue.
This central coordination improves the Quality-of-Service (QoS).
Indeed Figure \ref{fig:waiting_time_pooling} shows that the reduction in mean delay and delay probability can be substantial. \\

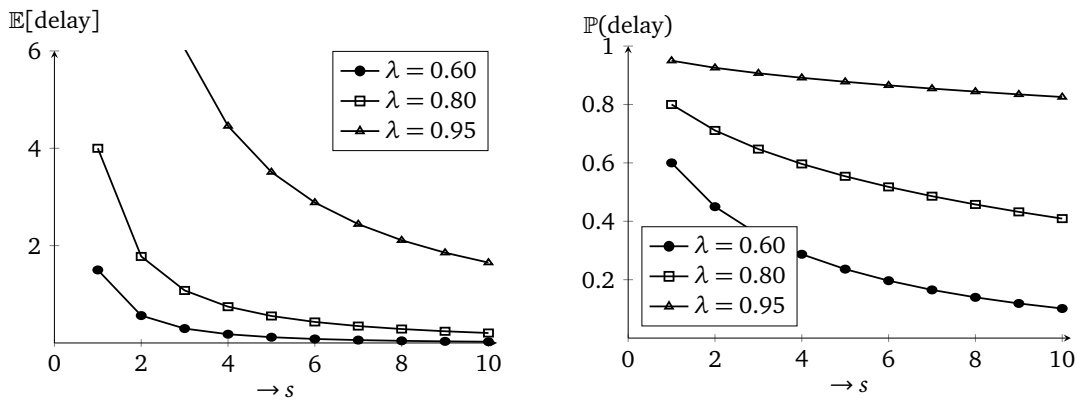
\begin{figure}[t]
\centering
\begin{subfigure}{0.48\textwidth}
\centering
\begin{tikzpicture}[scale=0.85]
\begin{axis}[
	xmin = 0, 
	xmax = 10.2,
	ymin = 0,
	ymax = 6,
	axis line style={->},
	axis x line = left,
	axis y line=middle,
	xlabel = {$\to s$},
	ylabel = {$\E[{\rm delay}]$},
	y label style = {at = {(axis cs: 0,6.2)},anchor = south},
	xscale=1,
	yscale=0.8,
	legend style = {at = {(axis cs: 10,6)},anchor = north east}]
	
\addplot[thick,mark=*] table{
1	1.5
2	0.5625
3	0.29562
4	0.179402
5	0.118076
6	0.0819026
7	0.0589477
8	0.0436069
9	0.0329511
10	0.0253248
};
\addplot[thick,mark=square] table {
1	4.
2	1.77778
3	1.07865
4	0.745541
5	0.554113
6	0.431477
7	0.347098
8	0.286028
9	0.240123
10	0.20459
};

\addplot[thick,mark=triangle] table {
1	19.
2	9.25641
3	6.04672
4	4.45709
5	3.5112
6	2.8853
7	2.44129
8	2.11042
9	1.85463
10	1.65117
};

\legend{{ $\lambda=0.60$},{$\lambda = 0.80$},{$\lambda = 0.95$}}
\end{axis}
\end{tikzpicture}
\end{subfigure}
\hspace{2mm}
\begin{subfigure}{0.48\textwidth}
\begin{tikzpicture}[scale=0.85]
\begin{axis}[
	xmin = 0, 
	xmax = 10.2,
	ymin = 0,
	ymax = 1.0,
	axis line style={->},
	axis x line = left,
	axis y line=middle,
	xlabel = { $\to s$},
	ylabel = { $\P({\rm delay})$},
	y label style = {at = {(axis cs: 0,1)},anchor = south},
	xscale=1,
	yscale=0.8,
	legend style={at={(0.03,0.05)},anchor=south west}]
	
\addplot[thick,mark=*] table {
1	0.6
2	0.45
3	0.354745
4	0.287043
5	0.236152
6	0.196566
7	0.165054
8	0.139542
9	0.118624
10	0.101299
};
\addplot[thick, mark=square] table {
1	0.8
2	0.711111
3	0.647191
4	0.596432
5	0.554113
6	0.517772
7	0.485938
8	0.457645
9	0.432222
10	0.40918
};
\addplot[thick, mark=triangle] table {
1	0.95
2	0.925641
3	0.907009
4	0.891419
5	0.877799
6	0.865589
7	0.854452
8	0.844169
9	0.834584
10	0.825586
};

\legend{{$\lambda=0.60$},{$\lambda = 0.80$},{$\lambda = 0.95$}}
\end{axis}

\end{tikzpicture}
\end{subfigure}
\caption{Effects of resource pooling in the $M/M/s$ queue.}
\label{fig:waiting_time_pooling}
\end{figure}

\vspace{.2cm}
 \noindent
\textbf{QED regime.}
The Quality-and-Efficiency driven (QED) regime is a form of resource pooling that goes beyond the typical objective of improving performance by joining forces. For the $M/M/s$ queue, the QED regime is best explained in terms of the square-root rule
\begin{equation}\label{beta}
s=\lambda+\beta\sqrt{\lambda}, \quad \beta>0,
\end{equation}
 which prescribes how to size capacity as a function of the offered load. Notice that the number of servers $s$ is taken equal to the sum of the mean load $\lambda$ and an additional term $\beta\sqrt{\lambda}$ that is of the same order as the natural load fluctuations of the arrival process (so of the order $\sqrt{\lambda}$). Observe that capacity increases with $\beta$, where we note that the free parameter $\beta$ can take any positive value. The QED regime assumes the coupling between $\lambda$ and $s$ as in \eqref{beta} and then lets both $s$ and $\lambda$ become large.
This not only increases the scale of operation, but also lets the load per server $\rho = \l/s \sim 1-\beta / \sqrt\l$ approach 1 as $s$ (and $\l$) become(s) large.
Now instead of diving immediately into the mathematical details, we shall first demonstrate the QED regime, or the capacity-sizing rule \eqref{beta}, by investigating typical sample paths of the queue length process $Q = (Q(t))_{t\geq 0}$  for increasing values of $\lambda$.

The upper middle panel of Figure \ref{fig:sample_paths_lambda10} depicts a sample path for $\lambda = 10$ and $s$ set according to \eqref{beta} rounded to the nearest integer.
The number of delayed jobs at time $t$ is given by $(Q(t)-s)^+$ with $(\cdot)^+ := \max\{0,\cdot\}$.
The number of idle servers is given by $(s-Q(t))^+$.
In Figure \ref{fig:sample_paths_lambda10}, the upper and lower area, enclosed between the sample path and the horizontal dashed line $s$, hence represent the cumulative queue length and cumulative number of idle servers, respectively, over the given time period.
Bearing in mind the dual goal of QoS and efficiency, we want to minimize both of these areas simultaneously.

We next show similar sample paths for increasing values of $\lambda$.
Since $s > \lambda$ is required for stability, the value of $s$ needs to be adjusted accordingly.
We show three scaling rules
\begin{equation}
\label{eq:intro_three_scaling_rules}
s^{(1)}_\lambda = \left[ \lambda + \beta \right ], \qquad
s^{(2)}_\lambda = [ \lambda + \beta\sqrt{\lambda} ], \qquad
s^{(3)}_\lambda = \left[ \lambda + \beta\,\lambda \right],
\end{equation}
with $\beta>0$, where $[\cdot]$ denotes the rounding operator.
Note that these three rules differ in terms of overcapacity $\sl-\lambda$, and $s_\lambda^{(2)}$ is the (rounded) square-root rule introduced in \eqref{beta}.
\begin{figure}[t]
\centering

\begin{tabular}{cccc} 
\rotatebox{90}{~~~~~~~~~~~~~~~~$\lambda=10$}
&
\includegraphics[scale=1.2]{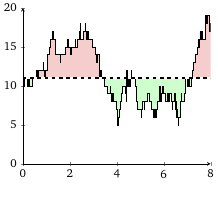}
&
\includegraphics[scale=1.2]{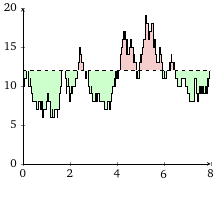}
&
\includegraphics[scale=1.2]{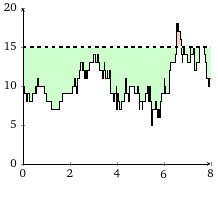}
\\
\rotatebox{90}{~~~~~~~~~~~~~~~~$\lambda=50$}
&
\includegraphics[scale=1.2]{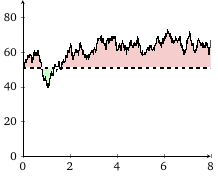}
&
\includegraphics[scale=1.2]{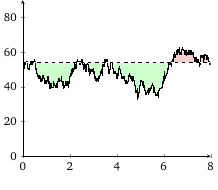}
&
\includegraphics[scale=1.2]{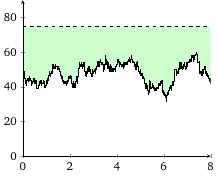}
\\
\rotatebox{90}{~~~~~~~~~~~~~~~~$\lambda=100$}
 &
\includegraphics[scale=1.2]{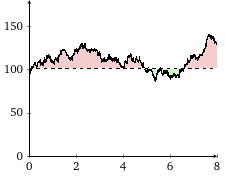}
&
\includegraphics[scale=1.2]{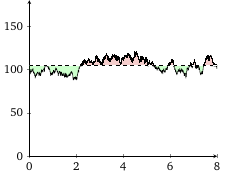}
&
\includegraphics[scale=1.2]{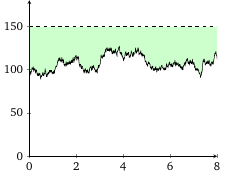}
\\
 &  \ \ $\sl^{(1)}$ & \ \ $\sl^{(2)}$ & \ \ $\sl^{3)}$  
 \end{tabular}

\caption{Sample paths of the $M/M/s$ queue with $\lambda = 10,\,50$ and $100$ and $s$ set according to the three scaling rules in \eqref{eq:intro_three_scaling_rules} with $\beta=0.5$.}
\label{fig:sample_paths_lambda10}
\end{figure}
Figure \ref{fig:sample_paths_lambda10} depicts typical sample paths of the queue length process for increasing values of $\lambda$ for the three scaling rules with $\beta = 0.5$.
Observe that for all scaling rules, the stochastic fluctuations of the queue length processes relative to $s$ decrease with the system size.
Moreover, the paths in Figure \ref{fig:sample_paths_lambda10} appear to become smoother with increasing $\l$.
Of course, the actual sample path always consists of upward and downward jumps of size 1, but we will show how proper centering and scaling of the queue length process indeed gives rise to a \textit{diffusion process} in the limit as $\lambda\to\infty$ (Section \ref{sec:three_examples}).
Although the difference in performance of the three regimes is not yet evident for relatively small $\lambda$, clear distinctive behavior occurs for large $\lambda$.

\vspace{.2cm}
 \noindent
\textbf{ED and QD regimes.}
With $s^{(1)}_\lambda$, most jobs are delayed and server idle time is low, since $\rho = (1+\beta/\lambda)^{-1} \to 1$ as $\lambda \to \infty$.
Systems scaled according to this rule value server efficiency over QoS and  therefore this regime is in the literature also known as the \textit{Efficiency-Driven} (ED) regime \cite{Zeltyn2005}.
In contrast, the third scaling rule $s^{(3)}_\l$ yields a constant utilization level $\rho = 1/(1+\beta)$, which stays away from 1, even for large $\lambda$.
Queues operating in this regime exhibit significant server idle times.
Moreover, for the particular realization of the queueing processes for $\lambda = 50$ and $\lambda=100$, none of the jobs is delayed.
This is known as the \textit{Quality-Driven} (QD) regime \cite{Zeltyn2005}.
The scaling rule $s^{(2)}_\l$ is in some ways a combination of the other two regimes.
First, we have $\rho = (1 +\beta/\sqrt{\lambda})^{-1} \to 1$ as $\lambda \to \infty$, which indicates efficient usage of resources as the system grows.
The sample paths, however, indicate that only a fraction of all jobs is delayed, and only small queues arise, indications of good QoS.
Figure \ref{fig:sample_paths_lambda10} provides visual confirmation that the square-root rule $s^{(2)}_\l$, related to the QED regime, strikes the right balance between the two profound objectives of capacity allocation in multi-server systems: negligible delay and idling.
We shall latter discuss the mathematical foundations of the QED regime and quantify the favorable properties revealed by Figure \ref{fig:sample_paths_lambda10}, including the non-degeneracy of the delay probability. To quote Halfin \& Whitt \cite{Halfin1981}: ``The balance between service and economy usually dictates that the probability of delay be kept away from both zero and one, so that the number of jobs present fluctuates between the regions above and below the number of servers.''



\vspace{.2cm}
\noindent
\textbf{Central Limit Theorem.}
We will see that not only the $M/M/s$ queue, but a wide range of multi-server models will possess the same property (delay probability being strictly between zero and one). This is because the QED regime is intimately connected with the Central Limit Theorem. Let  $\F$ denote the cumulative distribution function (cdf) of the standard normal distribution.
\begin{theorem}[Central Limit Theorem]\label{thm:CLT}
Let $X_1,X_2,\ldots$ be i.i.d.~random variables of finite mean $m$ and variance $v$. Then, for all $x\in\mathbb{R}$,
\begin{equation}
\lim_{n\to\infty}\P\Big(\sum_{i=1}^n X_i <mn +v x \sqrt{n}\Big)=\Phi(x).
\end{equation}
\end{theorem}

Consider a $\Pois(\l)$ random variable, with $\l$ integer-valued, which is equal in distribution to the sum of $\l$ independent Poisson random variables with unit mean and variance, i.e.
\[ \Pois(\l) \equalD \sum_{i=1}^\l \Pois_i(1). \]
Direct application of the CLT hence implies that for $s=\l+\beta\sqrt{\l}, $
\begin{equation}\label{eq:first_clt_poisson}
\P( \Pois(\l) \leq s) \to \Phi(\beta), \qquad \text{as }\l\to\infty.
\end{equation}


\vspace{.2cm}
\noindent
\textbf{Related surveys.}
In  this survey, we review the analysis of many-server systems operating in the QED regime, with special focus on various modeling assumptions that match well with the CLT.
In recent years, several comprehensible surveys have appeared in the literature on topics related to queueing systems and their asymptotic analysis.
We take the opportunity to mention a couple of them here.
Some tutorial papers are devoted to specific applications. Telephone call centers are the main focus of survey papers by Gans et al.~\cite{Gans2003}, Brown et al.~\cite{Brown2005} and Aksin et al.~\cite{Aksin2007}.
Armony et al.~\cite{Armony2015} provide an extensive overview of queueing phenomena in health care environments.
Focussing more on methodology, Pang et al.~\cite{Pang2007} discuss mathematical techniques to prove stochastic-process limits for queueing systems, and
Ward \cite{Ward2012} reviews queueing systems with abandonments in asymptotic regimes (including the QED regime). The survey paper by Dai and He~\cite{Dai2012} also concerns queueing systems  with abandonments, particularly focusing on the ED and QED regimes. {Whitt~\cite{Whitt2017} provides an extensive bibliography of the literature on queueing models with time-varying demand, also covering the QED regime.}

\vspace{.2cm}
\noindent\textbf{Organization.}
Section \ref{sec:three_examples} introduces two classical queueing models that serve as a vehicle to convey the ideas behind the QED regime.
 We discuss in Section~\ref{sec:properties} key properties that are common to these models under QED scaling, and illustrate how these features stretch beyond these specific model settings.
In Section \ref{sec:dim} we explain how asymptotic QED approximations of performance measures can be transformed into easy-to-use and robust capacity allocation principles.
Furthermore, we illustrate how to adapt capacity allocation decisions to time-varying demand.
Even though QED stochastic-process limits provide good first-order insight into the performance of large-scale systems, care needs to be taken with regard to the finite-ness of the system.
Therefore, we review in Section \ref{sec:refinements}  results that attempt to quantify the error made by asymptotic approximations, leading to both refinements and approximation bounds.
We also consider the implication of approximation errors for capacity allocation decisions (so-called optimality gaps).
Finally, in Section \ref{extsec} we review some model extensions  have received much attention due practical applicability or theoretical challenges.

\vspace{.2cm}
\noindent\textbf{Notation.}
We conclude this section by introducing some notation that will be used throughout the paper.
By $N(\mu,\sigma^2)$ we denote a normally distributed random variable with mean $\mu$ and variance $\sigma^2$.
The probability density function (pdf) and cumulative distribution function (cdf) of the standard normal distribution are denoted by $\f$ and $\F$, respectively.
The symbol $\equalD$ means equal in distribution and $\Rightarrowd$ means convergence in distribution.
The relation  $u(\l)\sim v(\l)$ implies that $\lim_{\l\to\infty} u(\l)/v(\l) = 1$.
By $u(\l) = O(v(\l))$ we mean that $\limsup_{\l\to\infty} u(\l)/v(\l)< \infty$, and $u(\l) = o(v(\l))$ implies that $\limsup_{\l\to\infty} u(\l)/v(\l) = 0$.

\section{Example models} \label{sec:three_examples}
This survey uses two running examples that are illustrative for both the model-specific and universal features of the QED regime. The first example is the already introduced $M/M/s$ queue, a fully Markovian many-server system. The second example is the so-called bulk-service queue, a standard discrete-time model. Through these models, we shall describe in this section several easy ways of establishing QED limits that only require a standard application of the CLT.

\subsection{Many exponential servers}\label{ss:ex1}

Let us first consider an infinite-server system to which jobs arrive according to a Poisson process with rate $\lambda$. Each jobs requires an exponentially distributed service time with unit mean. The steady-state number of jobs presents (or equivalently the steady-state number of busy servers) follows a Poisson distribution with mean $\lambda$. It is known that a Poisson distribution can be well approximated by a normal distribution for sufficiently large $\lambda$, so that it is approximately normally distributed with mean and variance $\lambda$. Therefore, the coefficient of variation (standard deviation divided by the mean) decreases as $1/\sqrt{\lambda}$, which makes the steady-state queue length become more concentrated around its mean with increasing $\lambda$.

If we now pretend, for a moment, that this infinite-server system serves as a good approximation for the $M/M/s$ queue, we could approximate the steady-state delay probability $\mathbb{P}({\rm delay})$ in the $M/M/s$ queue as
\begin{equation}\label{appis}
\mathbb{P}({\rm delay})\approx \mathbb{P}(Q\geq s)= \mathbb{P}\left(\frac{Q-\lambda}{\sqrt{\lambda}}\geq \frac{s-\lambda}{\sqrt{\lambda}}\right)
\approx 1-\Phi\left(\frac{s-\lambda}{\sqrt{\lambda}}\right)=1-\Phi(\beta).
\end{equation}
The use of this normal approximation in support of capacity allocation decisions was explored by Kolesar \& Green \cite{Kolesar1998}.
Of course, the infinite-server system ignores the one thing that makes a queueing system unique: that a queue is formed when all servers are busy. During these periods of congestion, a system with a finite number of servers $s$ will operate at a slower pace than its infinite-server counterpart, so the approximation in \eqref{appis} is likely to underestimate $\mathbb{P}({\rm delay})$. Nevertheless, the infinite-server heuristic does suggest that, in large systems, the number of servers can be chosen close to the offered load as in \eqref{beta}.

We shall now make more precise statements about QED limits, and use the intimate relation between the $M/M/s/s$ queue (Erlang loss model) and the $M/M/s$ queue (Erlang delay model).
When $\rho=\lambda/s<1$ the steady-state distribution of the $M/M/s$ queue exists and is given by
\begin{equation}
\label{eq:MMs_stationary_distribution}
\pi_k = \lim_{t\to\infty} \P( Q(t) = k )
= \left\{
\begin{array}{ll}
\pi_0\frac{\lambda^k}{k!}, & \text{if } k\leq s, \\
\pi_0\frac{\lambda^s}{s!}\,\rho^{k-s} & \text{if } k > s,
\end{array}
\right.
\end{equation}
where
\begin{equation*}
\pi_0 = \left( \sum_{k=0}^s \frac{\lambda^k}{k!} + \frac{\rho}{1-\rho} \frac{\lambda^s}{s!}\right)^{-1}.
\end{equation*}
From Little's law and the PASTA (Poisson Arrivals See Time Averages) property \cite{Wolff1982}, it follows that the delay probability, so the probability that an arbitrary job needs to wait before taken into service, is given by the Erlang C formula
\begin{equation}
\label{eq:MMs_wait}
C(s,\lambda) = \frac{\lambda^s}{s!} \left( (1-\rho) \sum_{k=0}^{s-1} \frac{\l^k}{k!} + \frac{\lambda^s}{s!} \right)^{-1}.
\end{equation}
The mean steady-state delay is given by
\begin{equation}
\label{eq:MMs_meanwait}
\E[{\rm delay}] = \frac{C(s,\lambda)}{(1-\rho)s}.
\end{equation}
A closely related performance measure is the probability of blocking in the $M/M/s/s$ queue, also known as the Erlang loss formula, and is given by
\begin{equation}
B(s,\lambda) = \frac{\frac{\lambda^s}{s!}}{\sum_{k=0}^{s} \frac{\l^k}{k!}} =\frac{\P(\Pois (\l) = s)}{ \P(\Pois(\lambda) < s) },
\end{equation}
where the latter probabilistic representation, with Pois$(\l)$ denoting a Poisson random variable with mean $\l$, is convenient in light of the CLT.
Note also that the Erlang B and C formulae are related by
\begin{equation}
\label{eq:proof_HW_0}
C(s,\lambda) = \left( \rho + \frac{1-\rho}{B(s,\lambda)} \right) ^{-1}.
\end{equation}
See \cite{WhittErlangBCoverview} for an extensive overview of properties of the Erlang B and C formulae; see also \cite{J74,Janssen2008b}. We now focus on how these formulae scale when $\lambda$ and $s$ both grow large.


Halfin \& Whitt \cite{Halfin1981} showed that, just as the tail probability in the infinite-server setting \eqref{appis}, the delay probability in the $M/M/s$ queue converges under scaling \eqref{beta} to a value between 0 and 1.
Moreover, they showed that this is in fact the only scaling regime in which such a non-degenerate limit exists and identified its value.

Let $\rho_\l := \l/\sl$ denote the the server utilization if capacity $\sl$ is scaled according to \eqref{beta}.
The following result is obtained in \cite{Halfin1981}.
\begin{proposition}
\label{prop:HalfinWhitt_delay_probability}
There is the  non-degenerate limit
\begin{equation}
\label{eq:HW_delay_prob}
\lim_{\lambda\to\infty} C(s_\lambda,\lambda ) = \left( 1+ \frac{\beta\,\F(\beta)}{\f(\beta)} \right)^{-1} =: g(\beta) \in (0,1)
\end{equation}
if and only if
\begin{equation}
\label{eq:HalfinWhitt_scaling}
\lim_{\lambda\to\infty} (1-\rho_\l) \sqrt{s_\lambda} \to \beta, \quad \beta > 0.
\end{equation}
In this case
\begin{equation}
\label{eq:HW_loss_prob}
\lim_{\lambda\to\infty} \sqrt{\lambda}\, B(s_\lambda,\lambda ) = \frac{\f(\beta)}{\F(\beta)}.
\end{equation}

\end{proposition}
\begin{proof}
%
%
Similar to \eqref{appis} we find
\begin{align}
\P(\Pois(\l) < \sl)
&= \P\left(\frac{\Pois(\l)-\l}{\sqrt{\l}} < \frac{\sl-\l}{\sqrt{\l}}\right)
= \P\left(\frac{\Pois(\l)-\l}{\sqrt{\l}} < (1-\rho_\l)\,\frac{\sl}{\sqrt\l}\right)\nonumber\\
&= \P\left(\frac{\Pois(\l)-\l}{\sqrt{\l}} < (1-\rho_\l)\,\sqrt{\sl}\left(1+o(1)\right) \right) \to \F(\beta),
\label{eq:proof_HW_1}
\end{align}
for $\l\to\infty$.
Stirling's formula gives
\begin{align}
\P(\Pois(\l)=s) &= {\rm e}^{-\l}\frac{\l^{\sl}}{\sl!}
\sim {\rm e}^{-\l} \l^{\sl}\cdot \frac{1}{\sqrt{2\pi\,\sl}} \left(\frac{\rm e}{\sl}\right)^{\sl} = \frac{1}{\sqrt{2\pi\sl}}\,\ee^{\sl-\l - \sl{\,\rm ln}(\rho_\l)}.
\end{align}
Since ${\,\rm ln}(\rho_\l) = -(1-\rho_\l) - \tfrac{1}{2}(1-\rho_\l)^2 + o((1-\rho_\l)^2)$ we find that
\begin{equation}
\label{eq:proof_HW_2}
\frac{ \P(\Pois(\l) = \sl) }{ 1-\rho_\l }
= \frac{1}{(1-\rho_\l)\sqrt{\sl}} \, \frac{\ee^{ -\tfrac{1}{2}(1-\rho_\l)^2\sl + o\left((1-\rho_\l)^2\sl\right)}}{\sqrt{2\pi}}  \to \frac{1}{\beta}\, \frac{\ee^{{-}\tfrac{1}{2} \beta^2}}{\sqrt{2\pi}} = \frac{\f(\beta)}{\beta}.
\end{equation}
Substituting \eqref{eq:proof_HW_1} and \eqref{eq:proof_HW_2} into \eqref{eq:proof_HW_0} gives \eqref{eq:HW_delay_prob}, and as by-product also \eqref{eq:HW_loss_prob}.
\end{proof}

\begin{figure}
\centering
\begin{tikzpicture}[scale=0.8]
\begin{axis}[
	xmin = 0, 
	xmax = 40,
	ymin = 0,
	ymax = 1.05,
	axis line style={->},
	axis lines = left,
	xlabel = $\to \l$,
	ylabel = {$C(\sl,\l)$},
	xscale=1,
	yscale=0.8]
\addplot[very thick] table
{
0.904875	0.904875
1.86349	0.898824
2.83172	0.895881
3.80494	0.894042
4.78134	0.892747
5.76	0.89177
6.74038	0.890997
7.72211	0.890366
8.70496	0.889838
9.68873	0.889386
10.6733	0.888995
11.6586	0.888651
12.6444	0.888346
13.6308	0.888073
14.6177	0.887826
15.605	0.887602
16.5927	0.887397
17.5807	0.887209
18.5691	0.887035
19.5578	0.886874
20.5467	0.886724
21.5359	0.886584
22.5254	0.886453
23.5151	0.88633
24.505	0.886214
25.4951	0.886105
26.4854	0.886002
27.4758	0.885904
28.4665	0.885811
29.4573	0.885722
30.4482	0.885638
31.4393	0.885558
32.4305	0.885481
33.4219	0.885407
34.4134	0.885337
35.405	0.885269
36.3967	0.885205
37.3885	0.885142
38.3805	0.885082
39.3725	0.885024
40.3647	0.884969
41.3569	0.884915
42.3492	0.884863
43.3417	0.884813
44.3342	0.884764
45.3267	0.884717
46.3194	0.884671
47.3122	0.884627
48.305	0.884584
49.2979	0.884543
};
\addplot[dashed] coordinates{	(0,0.880287) (40,0.880287) } ;
\addplot[very thick] table {
0.609612	0.609612
1.40693	0.581007
2.25	0.567757
3.11722	0.55969
4.	0.554113
4.89389	0.549958
5.79623	0.546707
6.70527	0.544073
7.6198	0.541882
8.53893	0.540021
9.46198	0.538416
10.3884	0.537012
11.3179	0.53577
12.25	0.534662
13.1845	0.533664
14.1211	0.532761
15.0597	0.531937
16.	0.531181
16.942	0.530485
17.8854	0.529842
18.8303	0.529244
19.7765	0.528687
20.7238	0.528166
21.6723	0.527678
22.6219	0.527219
23.5724	0.526786
24.5239	0.526378
25.4763	0.525991
26.4295	0.525624
27.3835	0.525275
28.3383	0.524944
29.2938	0.524627
30.25	0.524325
31.2068	0.524037
32.1643	0.52376
33.1224	0.523496
34.0811	0.523242
35.0403	0.522997
36.	0.522763
36.9603	0.522537
37.921	0.522319
38.8822	0.522109
39.8439	0.521906
40.806	0.52171
41.7686	0.521521
42.7315	0.521338
43.6949	0.521161
44.6586	0.520989
45.6228	0.520822
46.5873	0.520661
};
\addplot[dashed] coordinates{ (0,0.504539) (40,0.504539) } ;
\addplot[very thick] table {
0.381966	0.381966
1.	0.333333
1.69722	0.312006
2.43845	0.29945
3.20871	0.290969
4.	0.284761
4.80742	0.27997
5.62772	0.276131
6.45862	0.272967
7.29844	0.270303
8.1459	0.268019
9.	0.266035
9.85995	0.264289
10.7251	0.262738
11.5949	0.261349
12.4689	0.260095
13.3467	0.258956
14.228	0.257915
15.1125	0.256959
16.	0.256078
16.8902	0.255261
17.783	0.254502
18.6782	0.253794
19.5756	0.253132
20.4751	0.25251
21.3765	0.251926
22.2798	0.251375
23.1849	0.250854
24.0917	0.250361
25.	0.249893
25.9098	0.249449
26.8211	0.249026
27.7337	0.248622
28.6477	0.248237
29.5628	0.247869
30.4792	0.247516
31.3967	0.247178
32.3153	0.246854
33.235	0.246543
34.1557	0.246243
35.0774	0.245955
36.	0.245677
36.9235	0.245409
37.8479	0.24515
38.7732	0.244901
39.6993	0.244659
40.6261	0.244426
41.5538	0.2442
42.4822	0.243981
43.4113	0.243768
};
\addplot[dashed] coordinates{ (0,0.223361) (40,0.223361) };

\end{axis}

\node[right] at (6.9,3.85) { \small $g(0.1)$ };
\node[right] at (6.9,2.2) { \small $g(0.5)$ };
\node[right] at (6.9,1) { \small $g(1)$ };

\end{tikzpicture}
\caption{The delay probability $C(\sl,\l)$ with $\sl = [ \l + \beta \sqrt{\l} ]$ for $\beta = 0.1,\ 0.5,$ and 1 as a function of $\l$. The limiting values $g(\beta)$ are plotted as dashed lines.}
\label{fig:delay_probs_HW_MMs}
\end{figure}
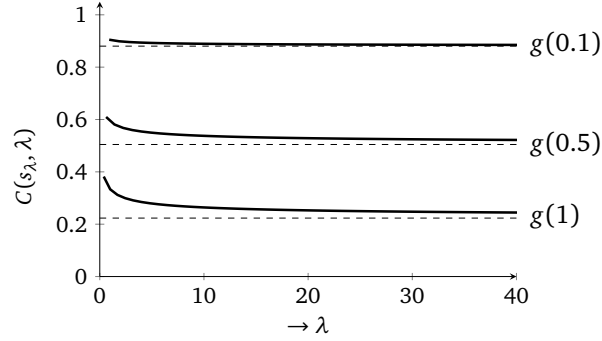

Many of the subsequent results in this survey presented for the $M/M/\sl$ queue can also be derived for the $M/M/\sl/\sl$ queue; we refer to \cite{Janssen2008b} for a detailed overview of these results.
Observe that $g(\beta)$ is a strictly decreasing function on $(0,\infty)$ with $g(\beta) \to 1$ as $\beta\to 0$ and $g(\beta)\to 0$ for $\beta\to\infty$.
Thus all possible delay probabilities are achievable in the QED regime, which will prove useful for the dimensioning of systems (see Section \ref{sec:dim}).
Although Proposition \ref{prop:HalfinWhitt_delay_probability} is an asymptotic result for $\l\to\infty$, Figure \ref{fig:delay_probs_HW_MMs} shows that $g(\beta)$ can serve as an accurate approximation for the delay probability for relatively small $\l$.
From Proposition \ref{prop:HalfinWhitt_delay_probability}, it also follows that under \eqref{eq:HalfinWhitt_scaling},
the limiting mean delay in \eqref{eq:MMs_wait} is given by
\begin{equation}
\label{eq:halfinwhitt_wait}
\frac{C(\sl,\lambda)}{(1-\rho_\l)\sqrt{\sl}} \to \frac{g(\beta)}{\beta}=: h(\beta), \quad \text{as }\l\to\infty.
\end{equation}
%
This implies that in the QED regime, the mean delay vanishes at rate $1/\sqrt{\sl}$ as $\l\to\infty$.
By Little's law this implies that the mean queue length is $O(\sqrt{\sl})$.
While these are all steady-state results, similar statements can be made for the entire queue-length process, as shown next.

\vspace{.2cm}
\noindent\textbf{Process-level convergence.}
QED scaling also gives rise to process-level limits,
where the evolution of the system occupancy, properly centered around $\sl$
and normalized by $\sqrt{\sl}$, converges to a diffusion process as $\l\to\infty$,
which again is fully characterized by the single parameter $\beta$.
This reflects that the system state typically hovers around the full-occupancy
level $\sl$, with natural fluctuations of the order $\sqrt{\sl}$.
Obtaining rigorous statements about stochastic-process limits poses considerable mathematical challenges.
Rather than presenting the deep technical details of the convergence results, we give a heuristic explanation of how the limiting process arises and what it should look like.

The queue-length process $Q^{(\sl)}(t)$ in Figure \ref{fig:sample_paths_lambda10} with scaling rule $\sl = [\l + \beta \sqrt{\l}]$ appears to concentrate around the level $\sl$.
As argued before, the stochastic fluctuations are of order $\sqrt{\l}$, or equivalently $\sqrt{\sl}$.
For that reason, we consider the centered and scaled process
\begin{equation}
\label{eq:intro_scaled_queue_length_process}
\bar Q^{(\sl)}(t) := \frac{ Q^{(\sl)}(t) - \sl}{\sqrt{\sl}}, \qquad \text{ for\ all } t\geq 0,
\end{equation}
and ask what happens to this process as $\l\to\infty$.
First, we consider the mean drift conditioned on $\bar Q^{(\sl)}(t) = x$.
When $x> 0$, this corresponds to a state in which $Q^{(\sl)}(t)>\sl$ and hence all servers are occupied.
Therefore, the mean rate at which jobs leave the system is $\sl$, while the arrival rate remains $\l$, so that the mean drift of $\bar Q^{(\sl)}(t)$ in $x>0$ satisfies
\begin{equation}
\frac{\l - \sl}{\sqrt{\sl}} \to -\beta, \qquad \text{as }\l\to\infty,
\end{equation}
under the scaling $\sqrt{\sl}(1-\rho_\l)\to \beta$ in \eqref{eq:HalfinWhitt_scaling}.
When $x\leq 0$, only $\sl + x\sqrt{\sl}$ servers are working, so that the net drift is
\begin{equation}
\frac{\l - (\sl + x\sqrt{\sl} )}{\sqrt{\sl}} \to -\beta-x, \qquad \text{as }\l\to\infty.
\end{equation}
Now, imagine what happens to the sample paths of $(\bar{Q}^{(\sl)}(t))_{t\geq 0}$ as we increase $\l$.
Within a fixed time interval, larger $\l$ and $\sl$ will trigger more and more events, both arrivals and departures.
Also, the jump size at each event epoch decreases as $1/\sqrt{\sl}$ as a consequence of the scaling in \eqref{eq:intro_scaled_queue_length_process}.
Hence, there will be more events, each with a smaller impact, and in the limit as $\l\to\infty$, there will be infinitely many events of infinitesimally small impact.
This heuristic explanation suggests that the process $\bar Q^{(\sl)}(t)$ converges to a stochastic-process limit, which is continuous, and has infinitesimal drift ${-}\beta$ above zero and ${-}\beta-x$ below zero.
Figure \ref{fig:sample_paths_diffusion} visualizes the emergence of the suggested scaling limit as $\l$ and $\sl$ increase.
\begin{figure}
\begin{subfigure}{0.32\textwidth}
\centering
\includegraphics[scale=0.8]{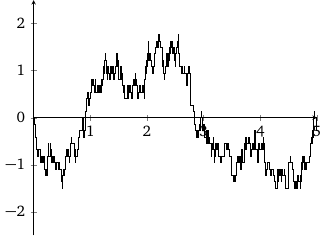}
\caption{$\l = 50$}
\end{subfigure}
\begin{subfigure}{0.32\textwidth}
\centering
\includegraphics[scale=0.8]{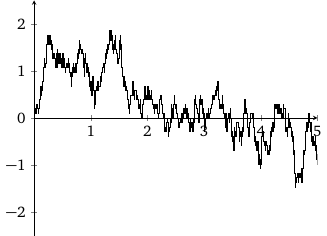}
\caption{$\l=100$}
\end{subfigure}
\begin{subfigure}{0.32\textwidth}
\centering
\includegraphics[scale=0.8]{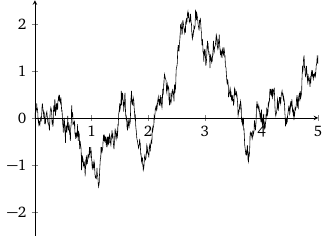}
\caption{$\l=500$}
\end{subfigure}
\caption{Sample paths of the normalized queue length process $\bar Q^{(\sl)}(t)$ with $\l = 50$, $\l=100$ and $\l=500$ and $\sl = [\l+0.5\sqrt{\l}]$.}
\label{fig:sample_paths_diffusion}
\end{figure}
The following theorem by Halfin \& Whitt \cite{Halfin1981} characterizes this scaling limit formally.
\begin{theorem}
\label{thm:Halfin_Whitt_diffusion}
Let $\bar Q^{(\sl)}(0)\, \Rightarrowd D(0) \in \mathbb{R}$ and $\sqrt{\sl}(1-\rho_\l)\to\beta$. Then for all $t\geq 0$,
\begin{equation}
\bar Q^{(\sl)}(t) \Rightarrowd D(t),\qquad \text{ as }\l\to\infty,
\end{equation}
where $D(t)$ is the diffusion process with infinitesimal drift $m(x)$ given by
\begin{equation}
m(x) = \left\{
\begin{array}{ll}
-\beta, & \text{if }x> 0,\\
-\beta-x, & \text{if } x \leq 0
\end{array}\right.
\label{222}
\end{equation}
and infinitesimal variance $\sigma^2(x) = 2$.
\end{theorem}
The limiting diffusion process $(D(t))_{t\geq 0}$ in Theorem \ref{thm:Halfin_Whitt_diffusion} is a combination of a negative-drift Brownian motion in the upper half plane and an Ornstein-Uhlenbeck process in the lower half plane.
We refer to this hybrid diffusion process as the Halfin-Whitt diffusion \cite{Leeuwaarden2012,Fralix2014,Browne1995}. Studying this diffusion process provides valuable information for the systems performance.

The fact that the properly centered and scaled occupancy process $(\bar Q^{(\sl)}(t))_{t\geq 0}$ has the weak limit $(D(t))_{t\geq 0}$, as stated in Theorem \ref{thm:Halfin_Whitt_diffusion}, has several important consequences. The boundary between the Brownian motion and the
Ornstein--Uhlenbeck process can be thought of as the number of servers, and
$(D(t))_{t\geq 0}$ will keep fluctuating between these two regions. The
process mimics a single-server queue above zero, and an infinite-server
queue below zero, for which Brownian motion and the Ornstein--Uhlenbeck process are indeed the respective heavy-traffic limits. As $\beta$ increases towards $+\infty$, capacity grows and the Halfin--Whitt diffusion will spend more time below zero.

The diffusion process $(D(t))_{t\geq 0}$ can thus be employed to obtain simple approximations for the system behavior.
Theorem \ref{thm:Halfin_Whitt_diffusion} supports approximating the occupancy process in the $M/M/\sl$ queue as
\begin{equation}
Q^{(\sl)}(\cdot)\stackrel{{\normalfont d}}{\approx} \sl+\sqrt{\sl} D(\cdot)
\end{equation}
when $\lambda$ and $\sl$ are large. It is natural to expect that this carries over to approximations for the steady-state
distribution of $(D(t))_{t\geq 0}$.
Let $D(\infty) := \lim_{t\to\infty} D(t)$ and $Q^{(\sl)}(\infty) := \lim_{t\to\infty} Q^{(\sl)}(t)$ denote the steady-state random variables. Then,
\begin{equation}\label{approx1}
Q^{(\sl)}(\infty)\stackrel{{\normalfont d}}{\approx} \sl+\sqrt{\sl} D(\infty).
\end{equation}
To rigorously justify the approximation \eqref{approx1} it is still required to show that the sequence
of steady-state distributions associated with the queue-length process, when appropriately scaled, converge to the
steady-state distribution associated with diffusion process,
\begin{equation}
\frac{Q^{(\sl)}(\infty)-s_\lambda}{\sqrt{s_\lambda}} \Rightarrowd D(\infty), \quad {\rm as} \ \lambda\to\infty.
\end{equation}
This has been done in \cite{Halfin1981}.

The steady-state characteristics of the diffusion were studied in \cite{Halfin1981}. Since the diffusion process $(D(t))_{t\geq 0}$ has piecewise linear drift, the procedure developed in \cite{Browne1995} to find the stationary distribution can be followed. This procedure consists of composing the density function as in \eqref{222} based on the density function of a Brownian motion with drift $-\beta$ for $x>0$ and of an Ornstein-Uhlenbeck process with drift $-\beta-x$ for $x<0$.
%
The density function of the stationary distribution for $(D(t))_{t\geq 0}$ is then proportional to $\f(x+\beta)/\F(\b)$ for negative levels $x<0$ and proportional to $\exp(\int_{0}^x m(u)\dd u)$ for $x\geq 0$.
Then, upon normalization, we find that
\begin{align}
\P(D(\infty) > 0 ) &= g(\beta), \label{eq:diff_1}\\
\P(D(\infty) \geq x | D(\infty) > 0) &= {\rm e}^{-\beta x} ,\quad \text{for }x>0,\\
\label{eq:diff:3}
\P(D(\infty) \leq x | D(\infty) \leq 0 ) &= \frac{\F(\beta+x)}{\F(\beta)},\quad \text{for }x\leq 0.
\end{align}
%
This confirms the earlier result for the Erlang C formula in \eqref{eq:HW_delay_prob}, i.e.
\begin{equation}
C(\sl,\lambda) \rightarrow \P( D(\infty) > 0 ) = g(\beta), \quad \text{as } \l\to\infty,
\end{equation}
and the scaled limiting mean delay in \eqref{eq:halfinwhitt_wait}
\begin{equation}
\frac{\E[ Q^{(\sl)}]}{\sqrt{\sl}} \rightarrow \E[D(\infty)] = \int_0^\infty g(\beta){\rm e}^{-\beta x} \dd x = \frac{g(\beta)}{\beta}, \quad \text{as } \l\to\infty.
\end{equation}
It is also of interest to study time-dependent characteristics like mixing times, time-dependent distributions and first passage times, to enhance our understanding of how the $M/M/\sl$ queue, behaves over various time and space scales. The mixing time is closely related to the spectral gap, which for the Halfin--Whitt diffusion $(D(t))_{t\geq 0}$ has been identified by Gamarnik \& Goldberg \cite{Gamarnik2013} building on the results of van Doorn \cite{Doorn1985} on the spectral gap of the $M/M/\sl$ queue. An alternative derivation of this spectral gap was presented in by \cite{Leeuwaarden2011,Leeuwaarden2012}, along with expressions for the Laplace transform over time, and the large-time asymptotics for the time-dependent density. First passage times to large levels corresponding to highly congested states were obtained in \cite{maglaraszeevi,Fralix2014}.

For obvious reasons, the QED regime is also referred to as the Halfin-Whitt regime, and both these names are used interchangeably in the literature.

\subsection{Bulk-service queue}\label{ss:ex2}

We next consider the bulk-service queue, a standard model for digital communication \cite{Bruneel1993}, but also many more applications among which wireless networks, road traffic, reservation systems, health care; see \cite[Chap.~2]{johanthesis} for an overview.
Although the bulk-service queue gives rise to a plain reflected random walk, and is not a multi-server queue, in the same sense as the $M/M/s$ queue, we explain below how these two models are connected.

Let jobs again arrive according to a Poisson process with rate $\l$, but now we discretize time, so the number of new arrivals per time period is given by
a Pois$(\l)$ random variable.
Let $Q^{(\sl)}_k$ denote the number of delayed jobs at the start of the $k^{\rm th}$ period and assume that the system is able to process $\sl$ jobs at the end of each period.
The queue length process can then be described by the Lindley-type recursion \cite{Lindley1952}
\begin{equation}
\label{eq:discrete_recursion}
Q^{(\sl)}_{k+1} = \max\{ 0,Q^{(\sl)}_k + \Pois_k(\l) - \sl \},
\end{equation}
with $Q^{(\sl)}_0 = 0$ and $(\Pois_k(\l))_{k\geq 0}$ i.i.d.~random variables.
The queue length process is thus characterized by a random walk with i.i.d.~steps of size
$(\Pois(\l)-\sl)$, with a reflecting barrier at zero. We can iterate the recursion in \eqref{eq:discrete_recursion} to find
\begin{align}
Q^{(\sl)}_{k+1} &= \max\left\{ 0 , Q^{(\sl)}_k + \Pois_k(\l)-\sl \right\} \nonumber\\
&= \max\left\{ 0 , \max\{ 0 , Q^{(\sl)}_{k-1} + (\Pois_{k-1}(\l)-\sl)\} + (\Pois_k(\l)-\sl)\} \right\}\nonumber \\
&= \max\left\{ 0 , (\Pois_k(\l)-\sl) , Q^{(\sl)}_{k-1} + (\Pois_k(\l)-\sl) + (\Pois_{k-1}(\l)-\sl)\right\}\nonumber \\
&= \max_{0\leq j\leq k} \Big\{ \sum_{i=1}^j (\Pois_{k-i}(\l)-\sl)\Big\}
\equalD  \max_{0\leq j\leq k} \Big\{ \sum_{i=1}^j (\Pois_i(\l)-\sl) \Big\},
\label{eq:max_randomwalk}
\end{align}
where the last equality holds in distribution due to the duality principle for random walks, see e.g.~\cite[Sec.~7.1]{Ross1996}.
Stability requires that the mean step size satisfies $\E[\Pois(\l) - \sl] = \l-\sl < 0$.
We use the shorthand notation for the partial sum $S_k := \sum_{i=1}^k (\Pois_i(\l)-\sl)$.
Let $Q^{(\sl)}:= \lim_{k\to\infty} Q^{(\sl)}_k$ denote the stationary queue length.
The probability generating function (pgf) of $Q^{(\sl)}$ can then be expressed in terms of the pgf of the positive parts of the partial sum:
\begin{equation}
\label{eq:Spitzers_identity}
\E[ z^{Q^{(\sl)}} ]
= \exp\Big\{  - \sum_{k=1}^\infty \frac{1}{k}\, (1- \E[z^{S_k^+}]) \Big\},\qquad |z|\leq 1.
\end{equation}
From \eqref{eq:Spitzers_identity} we obtain for the mean queue length and empty-queue probability the expressions
\begin{align}
\E[Q^{(\sl)}] &= \sum_{k=1}^\infty \frac{1}{k}\, \E[ S_k^+ ],\nonumber\\
\P(Q^{(\sl)}= 0 ) &= \exp\Big\{ -\sum_{k=1}^\infty \frac{1}{k}\, \P( S_k^+ > 0 ) \Big\}.
\label{eq:spitzer_expressions}
\end{align}

There is a connection between the bulk-service queue and the $M/D/s$ queue.
To see this, consider the number of queued jobs $Q^{(\sl)}(k)$ at time epochs $k=0,1,2,\ldots$.
The we set the period length equal to one service time.
The number of new arrivals per time period is then given by the sequence of i.i.d.~random variables $(\Pois_k(\l))_{k\geq 1}$.
At the start of the $k^{\rm th}$ period, $Q^{(\sl)}_k$ customers are waiting.
Since the service time of a customer is equal to the period length, all jobs that are in service at the beginning of the period will have left the system by time $k+1$.
This implies that $\min\{Q^{(\sl)}_k,\sl\}$ of the jobs that were queued at time $k$ are taken into service during period $k$.
These however cannot possibly have departed before the end of the period, due to their deterministic service times.
If $Q^{(\sl)}_k<\sl$, then additionally $\min\{ \Pois_k(\l) , \sl-Q^{(\sl)}_k \}$ of the new arrivals are taken into service.
This yields a total of $\Pois_k(\l)$ arrivals, and $\min\{Q^{(\sl)}_k+\Pois_k(\l),\sl\}$ departures from the queueing system during period $k$.
In total, this adds up to the Lindley recursion \eqref{eq:discrete_recursion}.
Hence, although the bulk-service queue is technically not a multi-server queue, it gives rise to a recursive relation that describes the $M/D/s$ queue.

The reason why we choose to explain the QED regime through the bulk-service queue is that the elementary random walk perspective allows for a rather direct application of the CLT.
To see this, let us ask ourselves what happens if $\l$ grows large using the square-root rule \eqref{beta}.
Since $\E[\Pois(\l)-\sl] = \l-\sl = -\beta\sqrt{\l} + o(\sqrt{\l})$, it makes sense to consider the scaled queue length process $\bar Q^{(\sl)}_k := Q^{(\sl)}_k/\sqrt{\l}$ for all $k\geq 0$, with scaled steps $Y_k^{(\sl)} := (\Pois_k(\l)-\sl)/\sqrt{\l}$.
Dividing both sides of \eqref{eq:max_randomwalk} by $\sqrt{\l}$ then gives
\begin{equation}
\bar Q^{(\sl)}_{k+1} = \max_{0\leq j\leq k} \Big\{ \sum_{i=1}^j Y^{(\sl)}_k \Big\}.
\end{equation}
Hence by the CLT
\begin{equation*}
Y^{(\sl)}_k = \frac{ A^{(\l)}_k - \sl }{\sqrt\l} = \frac{A^{(\l)}_k-\l}{\sqrt\l} - \beta \ \Rightarrowd \ Y_k \equalD N(-\beta,1),
\end{equation*}
for $\l\to\infty$.
So we expect the scaled queue length process to converge in distribution to a reflected random walk with normally distributed increments, i.e.~a reflected \textit{Gaussian random walk}.
Indeed, it is easily verified that \cite{Janssen2008a}
\begin{equation}
\bar Q^{(\sl)}_k \ \Rightarrowd \  M_{\beta,k} := \max_{0\leq j\leq k} \Big\{\sum_{i=1}^j Y_i \Big\}, \quad \text{as }\l\to\infty.
\end{equation}
Let $M_\beta:= \lim_{k\to\infty} M_{\beta,k}$ denote the all-time maximum of a Gaussian random walk.
\vspace{-2mm}
It can be shown that $M_\beta$ almost surely exists and that $\bar Q^{(\sl)}:= \lim_{k\to\infty} \bar Q^{(\sl)}$ $\Rightarrowd M_\beta$ for instance by \cite[Prop.~19.2]{Spitzer1964} and \cite[Thm.~X6.1]{Asmussen2003}.
The following theorem can be proved using a similar approach as in \cite{Jelenkovic2004}.

\begin{theorem}\label{thm:gaussian_rw}
If $(1-\rho_{\l})\sqrt{\l}\to\beta$ as $\l\to\infty$, then
\begin{enumerate}
\item[{\normalfont (i)}] $\bar Q^{(\sl)} \Rightarrowd M_\beta$ as $\l\to\infty$;
\item[{\normalfont (ii)}] $\P(\bar Q^{(\sl)} = 0) \to \P(M_\beta = 0)$ as $\l\to\infty$;
\item[{\normalfont (iii)}] $\E[\bar Q^{(\sl)\,k}] \to \E[M_\beta^k]$ as $\l\to\infty$ for any $k>0$.
\end{enumerate}
\end{theorem}

Hence, Theorem \ref{thm:gaussian_rw} is the counterpart of Theorem \ref{thm:Halfin_Whitt_diffusion}, but for the bulk-service (or $M/D/\sl$) queue, rather than for the $M/M/\sl$ queue.
Both theorems identify the stochastic-process limit in the QED regime:
for the $M/M/s$ queue this is the Halfin-Whitt diffusion, and for the bulk-service queue this is the Gaussian random walk.

The Gaussian random walk is well studied \cite{Siegmund1978,Chang1997,Janssen2006,Blanchet2006,Janssen2006} and there is an intimate connection with Brownian motion. The only difference, one could say, is that Brownian motion is a continuous-time process, whereas the Gaussian random walk only changes at discrete points in time.
If $(B(t))_{t\geq 0}$ is a Brownian motion with drift $-\beta <0$ and infinitesimal variance $\sigma^2$ and $(W(t))_{t \geq 0}$ is a random walk with $N(-\beta,\sigma^2)$ distributed steps and $B(0) = W(0)$, then $W$ can be regarded as the process $B$ embedded at equidistant time epochs.
That is, $W(t) \equalD B(t)$ for all $t\in\mathbb{N}^+$.
For the maximum of both processes this coupling implies
\begin{equation}
\max_{k\in \mathbb{N}^+} W(k) = \max_{k\in \mathbb{N}^+} B(k) \leq_{\rm st}
\max_{t\in \mathbb{R}^+} B(t),
\label{eq:max_inequality}
\end{equation}
where $\leq_{\rm st}$ denotes stochastic dominance.
This difference in maxima is visualized in Figure \ref{fig:BrownianMotion_vs_GaussianRW}.
It is known that the all-time maximum of Brownian motion with negative drift $-\mu$ and infinitesimal variable $\sigma^2$ has an exponential distribution with mean $\sigma/2\mu$ \cite{Harrison1985}.
Hence, \eqref{eq:max_inequality} implies that $M_\beta$ is stochastically upper bounded by an exponential random variable with mean $1/2\beta$.
\begin{figure}
\centering
\includegraphics[scale=1]{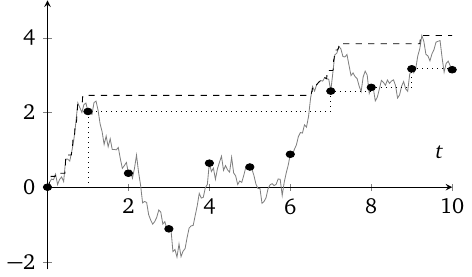}
\caption{Brownian motion (gray) and embedded Gaussian random walk (marked) with their respective running maxima (dashed and dotted, respectively).}
\label{fig:BrownianMotion_vs_GaussianRW}
\end{figure}

Despite this easy bound, precise results for $M_\beta$ are more involved. Let $\zeta$ denote the Riemann zeta function.
In \cite{Chang1997} and \cite{Janssen2006} it is shown that for $0<\beta<2\sqrt{\pi}$,
\begin{equation}
\P(M_\beta = 0) = \sqrt{2}\beta\, \exp \left\{ \frac{\beta}{\sqrt{2\pi}} \sum_{l=0}^\infty
\frac{\zeta(1/2-l)}{l!(2l+1)} \left(\frac{-\beta^2}{2}\right)^l \right\}
\end{equation}
and
\begin{equation}
\E[M_\beta] = \frac{1}{2\beta} + \frac{\zeta(1/2)}{\sqrt{2\pi}} + \frac{\beta}{4}
+ \frac{\beta^2}{\sqrt{2\pi}} \sum_{l=0}^\infty
\frac{\zeta(-1/2-l)}{l!(2l+1)(2l+2)} \left(\frac{-\beta^2}{2}\right)^l.
\end{equation}
In Figure \ref{fig:delay_probs_HW_Bulk}, we have plotted the exact empty-buffer probability and scaled mean delay, together with their asymptotic approximations.
We see that the performance measures associated with the Gaussian random walk serve as accurate approximations to performance measures describing the bulk-service queues of small to moderate size as well, just as we saw in Figure \ref{fig:delay_probs_HW_MMs} for the $M/M/\sl$ queue.

\begin{figure}
\centering
\begin{subfigure}{0.49\textwidth}
\begin{tikzpicture}[scale=0.7]
\begin{axis}[
	xmin = 0, 
	xmax = 40,
	ymin = 0,
	ymax = 1.05,
	axis line style={->},
	axis lines = left,
	xlabel = $\to \l$,
	ylabel = {$\P(Q^{(\sl)}=0)$},
	xscale=1,
	yscale=0.8]
\addplot[very thick] table {
0.381966	0.90552
1.	0.878709
1.69722	0.865297
2.43845	0.856953
3.20871	0.85114
4.	0.846798
4.80742	0.843399
5.62772	0.840646
6.45862	0.838358
7.29844	0.836418
8.1459	0.834746
9.	0.833286
9.85995	0.831997
10.7251	0.830848
11.5949	0.829815
12.4689	0.82888
13.3467	0.828028
14.228	0.827249
15.1125	0.826532
16.	0.825869
16.8902	0.825254
17.783	0.824681
18.6782	0.824147
19.5756	0.823646
20.4751	0.823175
21.3765	0.822732
22.2798	0.822314
23.1849	0.821918
24.0917	0.821543
25.	0.821187
25.9098	0.820849
26.8211	0.820526
27.7337	0.820218
28.6477	0.819924
29.5628	0.819643
30.4792	0.819374
31.3967	0.819115
32.3153	0.818867
33.235	0.818629
34.1557	0.818399
35.0774	0.818178
36.	0.817965
36.9235	0.81776
37.8479	0.817561
38.7732	0.817369
39.6993	0.817184
40.6261	0.817004
41.5538	0.816831
42.4822	0.816662
43.4113	0.816499
};
\addplot[dashed] coordinates{ (0,0.800543) (40,0.800543) } ;
\addplot[very thick] table {
0.609612	0.718204
1.40693	0.665638
2.25	0.640888
3.11722	0.625876
4.	0.615565
4.89389	0.607934
5.79623	0.602
6.70527	0.597216
7.6198	0.593256
8.53893	0.589909
9.46198	0.587031
10.3884	0.584523
11.3179	0.582312
12.25	0.580345
13.1845	0.578579
14.1211	0.576983
15.0597	0.575531
16.	0.574203
16.942	0.572982
17.8854	0.571854
18.8303	0.570809
19.7765	0.569837
20.7238	0.568929
21.6723	0.568079
22.6219	0.567281
23.5724	0.56653
24.5239	0.565822
25.4763	0.565152
26.4295	0.564517
27.3835	0.563915
28.3383	0.563342
29.2938	0.562797
30.25	0.562277
31.2068	0.56178
32.1643	0.561305
33.1224	0.56085
34.0811	0.560414
35.0403	0.559995
36.	0.559593
36.9603	0.559206
37.921	0.558833
38.8822	0.558474
39.8439	0.558127
40.806	0.557793
41.7686	0.557469
42.7315	0.557157
43.6949	0.556855
44.6586	0.556562
45.6228	0.556278
46.5873	0.556003
};
\addplot[dashed] coordinates{ (0,0.529325) (40,0.529325) } ;
\addplot[very thick] table {
0.904875	0.235113
1.86349	0.20218
2.83172	0.188182
3.80494	0.180078
4.78134	0.174664
5.76	0.170733
6.74038	0.167718
7.72211	0.165314
8.70496	0.163341
9.68873	0.161685
10.6733	0.160271
11.6586	0.159044
12.6444	0.157968
13.6308	0.157013
14.6177	0.15616
15.605	0.155391
16.5927	0.154694
17.5807	0.154057
18.5691	0.153473
19.5578	0.152936
20.5467	0.152438
21.5359	0.151976
22.5254	0.151546
23.5151	0.151143
24.505	0.150766
25.4951	0.150411
26.4854	0.150077
27.4758	0.149762
28.4665	0.149463
29.4573	0.14918
30.4482	0.148911
31.4393	0.148656
32.4305	0.148412
33.4219	0.14818
34.4134	0.147957
35.405	0.147745
36.3967	0.147541
37.3885	0.147346
38.3805	0.147158
39.3725	0.146978
40.3647	0.146804
41.3569	0.146637
42.3492	0.146476
43.3417	0.146321
44.3342	0.146171
45.3267	0.146026
46.3194	0.145885
47.3122	0.14575
48.305	0.145618
49.2979	0.145491
};
\addplot[dashed] coordinates{ (0,0.133419) (40,0.133419) };
\end{axis}

\node[right] at (6.9,3.5) { \small $\beta = 1$ };
\node[right] at (6.9,2.4) { \small $\beta = 0.5$ };
\node[right] at (6.9,0.65) { \small $\beta=0.1$ };

\end{tikzpicture}
\caption{Empty buffer probability}
\end{subfigure}
\begin{subfigure}{0.49\textwidth}
\begin{tikzpicture}[scale=0.7]
\begin{axis}[
	xmin = 0, 
	xmax = 40,
	ymin = 0,
	ymax = 5.01,
	axis line style={->},
	axis lines = left,
	xlabel = $\to \l$,
	ylabel = {$\E[Q^{(\sl)}]/\sqrt{\sl}$},
	xscale=1,
	yscale=0.8]
\addplot[very thick] table
{
0.381966	0.190983
1.	0.176741
1.69722	0.168792
2.43845	0.163645
3.20871	0.159981
4.	0.157209
4.80742	0.155019
5.62772	0.153234
6.45862	0.151743
7.29844	0.150473
8.1459	0.149375
9.	0.148414
9.85995	0.147563
10.7251	0.146804
11.5949	0.146119
12.4689	0.145499
13.3467	0.144934
14.228	0.144415
15.1125	0.143938
16.	0.143496
16.8902	0.143086
17.783	0.142704
18.6782	0.142347
19.5756	0.142012
20.4751	0.141697
21.3765	0.141401
22.2798	0.141121
23.1849	0.140856
24.0917	0.140605
25.	0.140366
25.9098	0.140139
26.8211	0.139922
27.7337	0.139716
28.6477	0.139518
29.5628	0.139329
30.4792	0.139148
31.3967	0.138975
32.3153	0.138808
33.235	0.138647
34.1557	0.138493
35.0774	0.138344
36.	0.138201
36.9235	0.138062
37.8479	0.137929
38.7732	0.1378
39.6993	0.137675
40.6261	0.137554
41.5538	0.137436
42.4822	0.137323
43.4113	0.137213
};
\addplot[dashed] coordinates{	(0,0.126373) (40,0.126373) } ;
\addplot[very thick] table {
0.609612	0.609612
1.40693	0.596739
2.25	0.588282
3.11722	0.58244
4.	0.57812
4.89389	0.574763
5.79623	0.572057
6.70527	0.569817
7.6198	0.567921
8.53893	0.56629
9.46198	0.564866
10.3884	0.56361
11.3179	0.56249
12.25	0.561483
13.1845	0.560572
14.1211	0.559742
15.0597	0.558982
16.	0.558282
16.942	0.557634
17.8854	0.557034
18.8303	0.556474
19.7765	0.555951
20.7238	0.55546
21.6723	0.554999
22.6219	0.554565
23.5724	0.554155
24.5239	0.553766
25.4763	0.553398
26.4295	0.553048
27.3835	0.552715
28.3383	0.552398
29.2938	0.552095
30.25	0.551805
31.2068	0.551528
32.1643	0.551262
33.1224	0.551007
34.0811	0.550762
35.0403	0.550527
36.	0.5503
36.9603	0.550081
37.921	0.54987
38.8822	0.549667
39.8439	0.54947
40.806	0.54928
41.7686	0.549096
42.7315	0.548918
43.6949	0.548746
44.6586	0.548578
45.6228	0.548416
46.5873	0.548259
};
\addplot[dashed] coordinates{ (0,0.532063) (40,0.532063) } ;
\addplot[very thick] table {
0.904875	4.52438
1.86349	4.51542
2.83172	4.50758
3.80494	4.50172
4.78134	4.49721
5.76	4.4936
6.74038	4.49064
7.72211	4.48815
8.70496	4.48602
9.68873	4.48417
10.6733	4.48254
11.6586	4.48109
12.6444	4.47979
13.6308	4.47862
14.6177	4.47756
15.605	4.47658
16.5927	4.47568
17.5807	4.47486
18.5691	4.47409
19.5578	4.47337
20.5467	4.4727
21.5359	4.47208
22.5254	4.47149
23.5151	4.47094
24.505	4.47041
25.4951	4.46992
26.4854	4.46945
27.4758	4.469
28.4665	4.46858
29.4573	4.46817
30.4482	4.46779
31.4393	4.46742
32.4305	4.46707
33.4219	4.46673
34.4134	4.4664
35.405	4.46609
36.3967	4.46579
37.3885	4.4655
38.3805	4.46522
39.3725	4.46495
40.3647	4.4647
41.3569	4.46445
42.3492	4.4642
43.3417	4.46397
44.3342	4.46374
45.3267	4.46352
46.3194	4.46331
47.3122	4.4631
48.305	4.4629
49.2979	4.46271
};
\addplot[dashed] coordinates{ (0,4.44199) (40,4.44199) };

\end{axis}

\node[right] at (6.9,4) { \small $\beta = 0.1$ };
\node[right] at (6.9,0.6) { \small $\beta = 0.5$ };
\node[right] at (6.9,0.1) { \small $\beta = 1$ };

\end{tikzpicture}
\caption{Scaled queue length}
\end{subfigure}
\caption{Delay probability and mean delay in the bulk-service queue with $\sl = \l + \beta \sqrt{\l}$ and $\beta=0.1,$ $0.5$ and $1$ as a function of $\l$. The asymptotic approximations are plotted as dashed lines.}
\label{fig:delay_probs_HW_Bulk}
\end{figure}
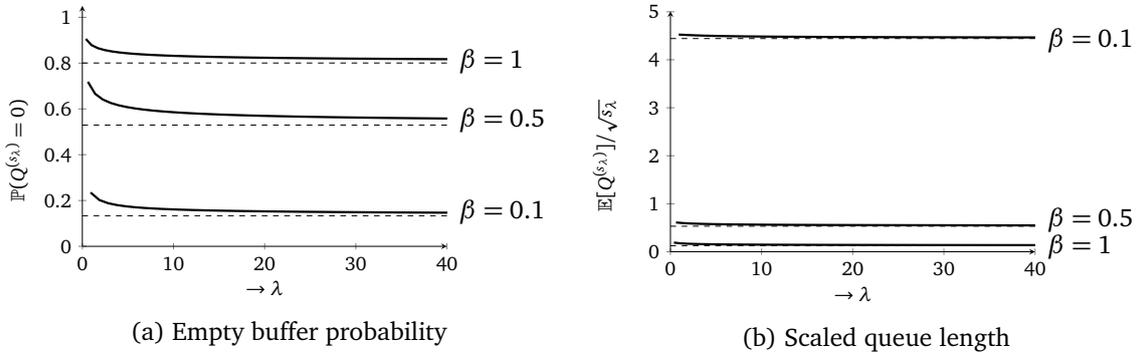

\section{Key QED properties}\label{sec:properties}

Now that we have seen how the square-root rule \eqref{beta} yields non-degenerate limiting behavior in classical queueing models, we shall summarize the revealed  QED properties and argue that these properties should hold for a more general class of models.
The first property relates to the efficient usage of resources, expressed as
\begin{equation}\label{eq:efficiency}
{\rm system\ load}
\sim 1-\frac{\rm constant}{\sqrt{\rm system\ size}}. \tag{Efficiency} \end{equation}
This property for the $M/M/\sl$ queue and bulk-service queue is a direct consequence of the square-root rule.
The second distinctive property is the balance between QoS and efficiency:
\begin{equation}\label{eq:balance}
\P({\rm delay}) \to {\rm constant},  \tag{Balance}
\end{equation}
as the system size increases indefinitely.
Indeed, we have shown that under \eqref{beta} and letting $\l,\sl\to\infty$ that both limiting functions $g(\b)$ in the $M/M/\sl$ queue and $\P(M_\beta>0)$ in the bulk-service queue can take all values in the interval $(0,1)$ by tuning the parameter $\beta$.
The third property relates to good QoS:
\begin{equation}\label{eq:QoS}
\E[{\rm delay}] = O(1/\sqrt{{\rm system\ size}}). \tag{QoS}
\end{equation}
Indeed, we have
\begin{equation}
\E[W^{(\sl)}] = \frac{h(\beta)}{\sqrt{\sl}} + o(1/\sqrt{\sl}) \qquad \text{and} \qquad \E[Q^{(\sl)}] = \sqrt{\sl}\,\E[M_\beta] + o(1/\sqrt{\sl}),
\end{equation}
in the $M/M/\sl$ queue and bulk-service queue, respectively.
Hence the mean delay vanishes at rate $1/\sqrt{\sl}$. Thus, an emerging appealing property of the QED regime is that the sojourn time of a customer is dominated by the magnitude of its service requirement. This is contrasting with the so-called non-degenerate slowdown regime, where the slowdown (the ratio between the sojourn time and the service time) is strictly larger than one  \cite{Atar2012}.

Since the mathematical underpinning of these properties comes from the CLT (as shown in Section \ref{sec:three_examples}), we can expect the properties to hold for a larger class of models.
We will illustrate this by discussing several extensions of the basic models discussed in Section \ref{sec:three_examples}.
The easiest way to do so seems to interpret the bulk-service queue as a many-sources model.
Consider a stochastic system in which demand per period is given by some random variable $A$, with mean $\mu_A$ and variance $\sigma_A^2<\infty$.
For systems facing large demand we propose to set the capacity according to the more general rule \[s = \mu_A + \beta\sigma_A,\] which consists of a minimally required part $\mu_A$ and a variability hedge $\beta\sigma_A$.
Assume that the demand is generated by $n$ stochastically identical and independent sources.
Each source $i$ generates $A_{i,k}$ work in the $k$th period, with $\E[A_{i,k}] = \mu$ and ${\rm Var}\,A_{i,k} = \sigma^2$.
Then the total amount of work arriving to the system during one period is $A_k^{(n)} = \sum_{i=1}^n A_{i,k}$ with mean $n\mu$ and variance $n\sigma^2$.
Assume that the system is able to process a deterministic amount of work $s_n$ per period and denote by $Q^{(n)}_k$ the amount of work left over at the end of period $k$.
Then,
\begin{equation}
 Q^{(n)}_{k+1} = \left( Q^{(n)}_k + A^{(n)}_k - s_n \right)^+.
 \end{equation}
Given that $s_n >  \E[A^{(n)}_1] = n\mu$, the steady-state limit $Q^{(n)} := \lim_{t\to\infty} Q^{(n)}(t)$ exists and satisfies \begin{equation} Q^{(n)} \equalD  \left( Q^{(n)} + A^{(n)}_k - s_n \right)^+.
\label{eq:bulk_service_stationary_recursion}
\end{equation}
With this many-sources interpretation \cite{Anick1982,Janssen2005,Janssen2008}, increasing the system size is done by increasing $n$, the number of sources.
As we have seen before, it requires a rescaling of the process $Q^{(n)}$ by an increasing sequence $c_n$, to obtain a non-degenerate scaling limit $Q := \lim_{n\to\infty} Q^{(n)}/c_n$.
(We omit the technical details needed to justify the interchange of limits.) From \eqref{eq:bulk_service_stationary_recursion} it becomes clear that the scaled increment
\begin{equation}
\frac{A^{(n)}_k - s_n}{c_n} = \frac{\sum_{i=1}^n A_{i,k} - n\mu}{c_n} + \frac{n\mu - s_n}{c_n}
\end{equation}
only admits a proper limit if $c_n$ is of the form $c_n = O(\sqrt{n})$, by virtue of the CLT, and $(s_n-n\mu)/c_n \to \beta >0$ as $n\to\infty$.
Especially for $c_n = \sigma\sqrt{n}$, the standard deviation of the demand per period, this reveals that $Q$ has a non-degenerate limit, which is equal in distribution to the maximum of a Gaussian random walk with drift ${-}\beta$ and variance 1, if
\[ s_n = n\mu+\beta \sigma \sqrt{n} + o(\sqrt{n}).
\]
Moreover, the results for the Gaussian random walk presented in Section \ref{ss:ex2} are applicable to this model and the key features of the QED scaling carry over to this more general setting.
That is, for the bulk-service queue under the general assumptions above we get the QED approximation
\begin{equation}
\E[Q^{(n)}] \approx \sigma\sqrt{n}\,\E[M_\beta] \approx \frac{\sigma\sqrt{n}}{2\beta}
\end{equation}
for small $\beta$.
Thus, the many-sources framework shows that the QED scaling finds much wider application than just queueing models with Poisson input.

Let us reflect on a key technical difference between the bulk-service queue and the $M/M/s$ queue. The bulk-service queue is and remains a one-dimensional reflected random walk, even under the QED scaling. Therefore, to establish the QED limits for the performance measure, one only needs to apply the CLT to the increments of the random walk, which readily shows that the queue converges to the Gaussian random walk. Analysis of multi-server queues is typically more challenging. Establishing QED limits for the elementary $M/M/s$ queue already contains some technically advanced steps. While we explained the high-level insights to argue the convergence of the birth-death process taking discrete steps to the continuous diffusion process, the formal proof in Halfin \& Whitt \cite{Halfin1981} relies on Stone's Theorem \cite{stone1963limit,iglehart1974weak,kou2004diffusion} for the weak convergence of birth-death processes to diffusion processes. However, for multi-server queue that cannot be viewed as a birth-death process, Stone's Theorem cannot be applied and entirely different techniques are needed; see Subsection \ref{sub:general}.

\section{Dimensioning}
\label{sec:dim}
We adopt the term \textit{dimensioning} used by Borst et al.~\cite{Borst2004} to say that the capacity of a system is adapted to the load in order to reach certain performance levels.
In \cite{Borst2004} dimensioning refers to the staffing problem in a large-scale call center and key ingredients are the square-root rule in \eqref{beta} and the QED regime.
We now revisit the results in \cite{Borst2004} and its follow-up works to explain this connection to the QED regime. {We also  discuss the time-varying setting in which jobs arrive according to a non-homogeneous Poisson process, and  the capacity is dynamically adapted to the load.}

\subsection{Constraint satisfaction}
\label{sec:intro_constraint}
Consider the $M/M/s$ queue with arrival rate $\l$ and service rate $\mu=1$.
A classical dimensioning problem is to determine the minimum number of servers $s$ necessary to achieve a certain target level of service, say in terms of delay.

Suppose we want to determine the minimum number of servers such that the fraction of jobs that are delayed in the queue is at most $\varepsilon\in(0,1)$.
Hence we should find
\begin{equation}\label{eq:tagA}
s^{*}_\l(\eps) := \min \left\{s > \l\, |\, C(s,\l) \leq \eps \right\}.
\end{equation}
But alternatively, we can use the QED framework, which says that with $\sl$ as in \eqref{beta}, \\* $\lim_{\l\to\infty} C(\sl,\l) = g(\beta)$ (see Proposition \ref{prop:HalfinWhitt_delay_probability}).
Then \eqref{eq:tagA} can be replaced by
\begin{equation}
s^{\rm QED}_\l(\eps) = \lceil \l + \beta^*(\eps) \sqrt{\l}\rceil,
\end{equation}
where $\beta^*(\eps)$ solves
\begin{equation}
g(\beta^*) = \eps.
\end{equation}
In Figure \ref{fig:MMs_staffing_levels} we plot the exact (optimal) capacity level $s^*_\l(\eps)$ and the heuristically obtained capacity level $s^{\rm QED}_\l(\eps)$ as functions of $\eps$ for several loads $\l$.

\begin{figure}
\centering
\begin{subfigure}{0.32\textwidth}\centering
\begin{tikzpicture}[scale = 0.6]
\Large
\begin{axis}[
	xmin = 0,
	xmax = 1,
	ymin = 10,
	ymax = 19,
	axis line style={->},
	axis lines = left,
	legend cell align = left,
	xlabel = { $\to \eps$},
	ylabel = {},
	yscale = 0.8,
	legend style = {at = {(1,1.2)}, anchor = north east}]
	
\addplot[very thick] file {tikz/Constraint_Satisfaction/lambda10_exact.txt};
\addplot[very thick, dashed, col1] file {tikz/Constraint_Satisfaction/lambda10_asymptotic.txt};
\legend{{$s^*_\l(\eps)$},$s^{\rm QED}_\l(\eps)$}
\end{axis}
\end{tikzpicture}
\caption{$\l=10$}
\end{subfigure}
\begin{subfigure}{0.32\textwidth}\centering
\begin{tikzpicture}[scale = 0.6]
\Large
\begin{axis}[
	xmin = 0,
	xmax = 1,
	ymin = 100,
	ymax = 125,
	axis line style={->},
	axis lines = left,
	legend cell align = left,
	xlabel = {$\to \eps$},
	ylabel = {},
	yscale = 0.8,
	legend style = {at = {(1,1.2)}, anchor = north east}]
	
\addplot[very thick] file {tikz/Constraint_Satisfaction/lambda100_exact.txt};
\addplot[very thick, dashed, col1] file {tikz/Constraint_Satisfaction/lambda100_asymptotic.txt};
\legend{{$s^*_\l(\eps)$},$s^{\rm QED}_\l(\eps)$}
\end{axis}
\end{tikzpicture}
\caption{$\l=100$}
\end{subfigure}
\begin{subfigure}{0.33\textwidth}\centering
\begin{tikzpicture}[scale = 0.6]
\Large
\begin{axis}[
	xmin = 0,
	xmax = 1,
	ymin = 500,
	ymax = 550,
	axis line style={->},
	axis lines = left,
	legend cell align = left,
	xlabel = {$\to \eps$},
	ylabel = {},
	yscale = 0.8,
	legend style = {at = {(1,1.2)}, anchor = north east}]
	
\addplot[very thick] file {tikz/Constraint_Satisfaction/lambda500_exact.txt};
\addplot[very thick, dashed, col1] file {tikz/Constraint_Satisfaction/lambda500_asymptotic.txt};
\legend{{$s^*_\l(\eps)$},$s^{\rm QED}_\l(\eps)$}
\end{axis}
\end{tikzpicture}
\caption{$\l=500$}
\end{subfigure}
\caption{Capacity levels as a function of the delay probability targets $\eps$.}
\label{fig:MMs_staffing_levels}
\end{figure}
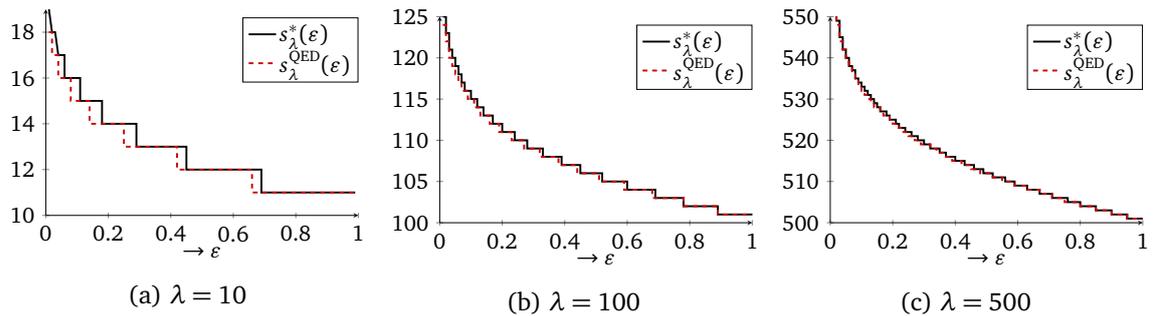

Observe that even for very small values of $\l$, the capacity function $s^{\rm QED}(\eps)$ coincides with the exact solution for almost all $\eps\in(0,1)$ and differs no more than by one server for all $\eps$.
Borst et al.~\cite{Borst2004} recognized this in their numerical experiments too, and \cite{Janssen2011} later confirmed this theoretically (see Section \ref{sec:dim}).
One can easily formulate other constraint satisfaction problems and reformulate them in the QED regime.
For instance, constraints on the mean delay or the tail probability of the duration of delay, e.g.~$\P({\rm delay}>T)$, which are asymptotically approximated by $h(\beta)/\sqrt{\l}$ and $g(\beta)\ee^{-\beta \sqrt{\l} T}$, respectively.
See \cite{Borst2004,Zhang2012,Sanders2014} for more examples.

\subsection{Cost minimization}
\label{sec:intro_optimization}
Alternatively, one can consider optimization problems, for instance to strike the right balance between the capacity allocation costs and delay costs incurred.
More specifically, assume an allocation cost of $a$ per server per unit time, and a penalty cost of $q$ per delayed job per unit time, yielding the total cost function
\[
\bar{K}(s,\l) := a\,s + q\,\l\E[{\rm delay}] = a\,s + q\l\,\frac{C(s,\l)}{s-\l},
\]
see \eqref{eq:MMs_meanwait}, and then ask for the capacity level $s$ that minimizes $\bar{K}(s,\l)$.
Since $s>\l$, we have $\bar{K}(s,\l) > a\,\l$ for all feasible solutions $s$.
Moreover, the minimizing value of $\bar{K}(s,\l)$ is invariant with respect to scalar multiplication of the objective function.
Hence we equivalently seek to optimize
\begin{equation}
\label{eq:optimization_objective}
K(s,\l) = r\,(s-\l) + \frac{\l}{s-\l}\,C(s,\l) \qquad \text{with } r = a/q.
\end{equation}
Denote by $s^*_\l(r) := \arg\min_{s > \l} K(s,\l)$ the true optimal capacity level.
With $\sl = \l + \beta\sqrt{\l}$ and the QED limit in \eqref{eq:halfinwhitt_wait}, we can replace \eqref{eq:optimization_objective} by its asymptotic counterpart:
\begin{align}
\frac{K(\sl,\l)}{\sqrt{\l}} \to r\,\beta + \frac{g(\beta)}{\beta} =: K_*(\beta), \qquad \text{as }\l\to\infty.
\label{eq:total_costs}
\end{align}
We again obtain a limiting objective function that is easier to work with than its exact pre-limit counterpart.
Hence, in the spirit of the asymptotic resource allocation procedure in the previous subsection, we propose the following method to determine the capacity level that minimizes overall costs.
First, (numerically) compute the value $\beta^*(r) = \arg\min_{\beta>0} K_*(\beta)$, which is well-defined, because the function $K_*(\beta)$ is strictly convex for $\beta>0$.
Then, set $s^{\rm QED}_\l(r) = [ \l + \beta^*(r) \sqrt{\l} ]$.
In Figure \ref{fig:MMs_staffing_levels_optimization} we compare the outcomes of this asymptotic resource allocation procedure against the true optima as a function of $r\in(0,\infty)$, for several values of $\l$.
The capacity levels $s^{\rm QED}_\l(r)$ and $s^*_\l(r)$ are aligned for almost all $r$, and differ no more than one server for all instances.

\begin{figure}
\centering
\begin{subfigure}{0.32\textwidth}\centering
\begin{tikzpicture}[scale = 0.6]
\Large
\begin{axis}[
	xmin = 0,
	xmax = 1,
	ymin = 10,
	ymax = 19,
	axis line style={->},
	axis lines = left,
	legend cell align = left,
	xlabel = { $\to \eps$},
	ylabel = {},
	yscale = 0.8,
	legend style = {at = {(1,1.2)}, anchor = north east}]
	
\addplot[very thick] file {tikz/Optimization/lambda10_exact.txt};
\addplot[very thick, dashed, col1] file {tikz/Optimization/lambda10_asymptotic.txt};
\legend{{$s^*_\l(\eps)$},$s^{\rm QED}_\l(\eps)$}
\end{axis}
\end{tikzpicture}
\caption{$\l=10$}
\end{subfigure}
\begin{subfigure}{0.32\textwidth}\centering
\begin{tikzpicture}[scale = 0.6]
\Large
\begin{axis}[
	xmin = 0,
	xmax = 1,
	ymin = 100,
	ymax = 125,
	axis line style={->},
	axis lines = left,
	legend cell align = left,
	xlabel = {$\to \eps$},
	ylabel = {},
	yscale = 0.8,
	legend style = {at = {(1,1.2)}, anchor = north east}]
	
\addplot[very thick] file {tikz/Optimization/lambda100_exact.txt};
\addplot[very thick, dashed, col1] file {tikz/Optimization/lambda100_asymptotic.txt};
\legend{{$s^*_\l(\eps)$},$s^{\rm QED}_\l(\eps)$}
\end{axis}
\end{tikzpicture}
\caption{$\l=100$}
\end{subfigure}
\begin{subfigure}{0.33\textwidth}\centering
\begin{tikzpicture}[scale = 0.6]
\Large
\begin{axis}[
	xmin = 0,
	xmax = 1,
	ymin = 500,
	ymax = 550,
	axis line style={->},
	axis lines = left,
	legend cell align = left,
	xlabel = {$\to \eps$},
	ylabel = {},
	yscale = 0.8,
	legend style = {at = {(1,1.2)}, anchor = north east}]
	
\addplot[very thick] file {tikz/Optimization/lambda500_exact.txt};
\addplot[very thick, dashed, col1] file {tikz/Optimization/lambda500_asymptotic.txt};
\legend{{$s^*_\l(\eps)$},$s^{\rm QED}_\l(\eps)$}
\end{axis}
\end{tikzpicture}
\caption{$\l=500$}
\end{subfigure}
\caption{Optimal capacity levels as a function of $r = a/q$.}
\label{fig:MMs_staffing_levels_optimization}
\end{figure}
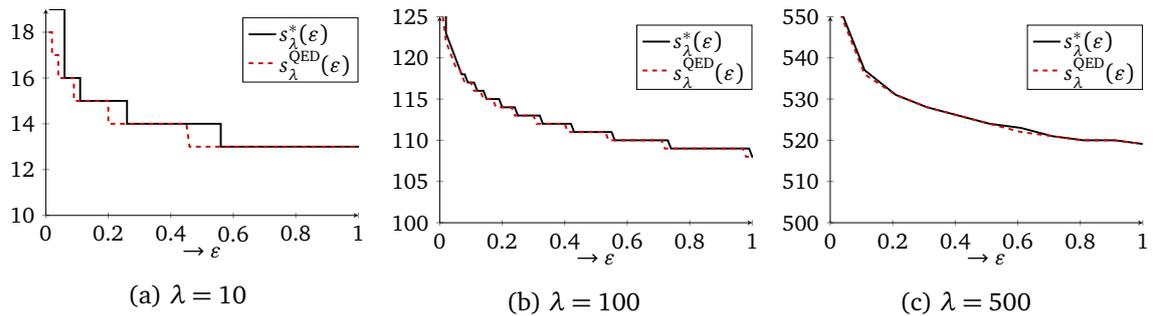

\subsection{{Dynamic rate queues}}


We next discuss how the QED regime also finds application in systems facing a time-varying load.
A time-varying arrival rate $\l(t)$ calls for a time-varying capacity rule $s(t)$.
Again, we shall explain the main ideas through the $M/M/s$ queue, but now its time-varying extension in which jobs arrive according to a non-homogeneous Poisson process with rate function $\l(t)$, a setting typically referred to as the $M_t/M/s_t$ queue.

As in  Section \ref{sec:intro_constraint} we want to set the capacity level $s(t)$ such that the delay probability is at most $\varepsilon\in(0,1)$ for all $t$. The analysis of this time-varying many-server queueing systems is cumbersome and several approximative analysis have been proposed such as the \textit{pointwise-stationary approximation} (PSA) \cite{Green1991}, which evaluates the system at time $t$ as if it were in steady-state with instantaneous parameters $\l=\l(t)$, $\mu$ and $s = s(t)$.
PSA performs well in slowly varying environments with relatively short service times \cite{Green1991,Whitt1991}, but the steady-state approximation becomes less accurate when $\lambda(t)$ displays significant fluctuations; see the numerical experiment at the end of this section.
One reason for this lack of accuracy is that PSA does not account for the jobs that are actually present in the system (being in service or queued), an important piece of real-time information that should be taken into account in capacity allocation decisions.
Jennings et al.~\cite{Jennings1996} introduced an alternative to PSA that exploits the relation with infinite-server queues, facing a non-homogeneous Poisson process with rate $\l(t)$, in which case the number of jobs at time $t$ is Poisson distributed with mean
\begin{equation}
\label{eq:offered_load_eick}
R(t) = \E\left[ \l(t-B)\right] \E[B] = \int_0^\infty \l(t-u)\,\P(B>u)\, {\rm d}u = \int_0^\infty \l(t-u)\, \ee^{-\mu u} \,{\rm d}u,
\end{equation}
where $B$ denotes the processing time of one jobs, in our case an exponentially distributed random variable.
We remark that under general service time assumptions, we should replace $\E[\l(t-B)]$ in \eqref{eq:offered_load_eick} with $\E[\l(t-B_e)]$, where $B_e$ denotes the excess service time \cite{Eick1993}.
Recall that the mean delay in the QED regime is negligible; see \eqref{eq:QoS}.
Hence, the total time in the system is roughly equal to its service time.
Under these conditions, the many-server system can be approximated by the infinite-server approximation with offered load as in \eqref{eq:offered_load_eick}.
Accordingly, we can determine the capacity levels $s(t)$ for each $t$ based on steady-state $M/M/s$ measures with offered load $R=R(t)$.
Jennings et al.~\cite{Jennings1996} proceed by exploiting the heavy-traffic results of Halfin \& Whitt \eqref{eq:halfinwhitt_wait}.
In conjunction with the dimensioning scheme in Section \ref{sec:intro_constraint}, it is proposed in \cite{Jennings1996} to set
\begin{equation}\label{eq:time_varying_s}
s(t) = \big\lceil R(t) + \beta^*(\eps) \sqrt{R(t)} \big\rceil,
\end{equation}
where $\beta^*(\eps)$ solves $g(\beta^*(\eps)) = \eps$.
Remark that the number of servers is rounded up to ensure that the achieved delay probability is indeed below $\eps$.
The time-dependent dimensioning rule in \eqref{eq:time_varying_s} was dubbed in \cite{Jennings1996,Massey1994} the modified offered load (MOL) approximation.
Let us now demonstrate how MOL works for an example with sinusoidal arrival rate function.
Figure \ref{fig:intro_example_arrival_a} shows an arrival rate pattern $\l(t)$ and corresponding offered load function $R(t)$ for $\mu=1/2$.
The resulting time-varying capacity levels based on the PSA and MOL approximations with $\eps = 0.3$ are plotted in
Figure \ref{fig:intro_example_arrival_b}.
Through simulation, we evaluate the delay probability as a function of time for $\eps = 0.1,\, 0.3$ and 0.5.
While the PSA approach fails to stabilize the performance of the queue, the MOL method does stabilize around the target performance, see Figure \ref{fig:time_varying_delay}.
The slightly erratic nature of the delay probability as a function of time can be explained by rounding effects of the capacity level.

\begin{figure}
\centering
\begin{subfigure}{0.32\textwidth}
\centering
\begin{tikzpicture}[scale = 0.66]
\begin{axis}[
	xmin = 0,
	xmax = 24,
	ymin = 0,
	ymax = 55,
	axis line style={->},
	axis lines = left,
	xlabel = {$\to t$},
	yscale = 0.8,
	legend cell align = left,
	legend style = {at = {(axis cs: 23.5,1)},anchor = south east}
	]
	
\addplot[very thick] file {tikz/TimeVarying/arrival_rate.txt};
\addplot[very thick, col1] file {tikz/TimeVarying/offered_load.txt};
\legend{{$\l(t)$},$R(t)$}
\end{axis}
\end{tikzpicture}
\caption{Arrival rate and offered load}
\label{fig:intro_example_arrival_a}
\end{subfigure}
\begin{subfigure}{0.32\textwidth}
\centering
\begin{tikzpicture}[scale = 0.66]
\begin{axis}[
	xmin = 0,
	xmax = 24,
	ymin = 0,
	ymax = 60,
	axis line style={->},
	axis lines = left,
	xlabel = {$\to t$},
	yscale = 0.8,
	legend cell align = left,
	legend style = {at = {(axis cs: 23.5,2)},anchor = south east}
	]
	
\addplot[thick] file {tikz/TimeVarying/s_PSA.txt};
\addplot[thick, col1] file {tikz/TimeVarying/s_Jennings.txt};
\legend{PSA,MOL}
\end{axis}
\end{tikzpicture}
\caption{Capacity levels}
\label{fig:intro_example_arrival_b}
\end{subfigure}
\begin{subfigure}{0.32\textwidth}
\begin{tikzpicture}[scale = 0.66]
\begin{axis}[
	xmin = 0,
	xmax = 24,
	ymin = 0,
	ymax = 1,
	axis line style={->},
	axis lines = left,
	xlabel = {$\to t$},
	ylabel = {},
	yscale = 0.8,
	legend style = {at = {(axis cs: 23.5,0.98)},anchor = north east}]
	
\addplot[thick, black] file {tikz/TimeVarying/pdelay_e01_mol.txt};

\addplot[thick, col1] file {tikz/TimeVarying/pdelay_e03_mol.txt};

\addplot[thick, col4] file {tikz/TimeVarying/pdelay_e05_mol.txt};
\legend{{$\eps=0.1$},{$\eps=0.3$},{$\eps=0.5$}}
\end{axis}
\end{tikzpicture}
\caption{Delay probabilities}
\label{fig:time_varying_delay}
\end{subfigure}
\caption{Time-varying parameters for the example with sinusoidal arrival rate}
\label{fig:intro_example_arrival}
\end{figure}
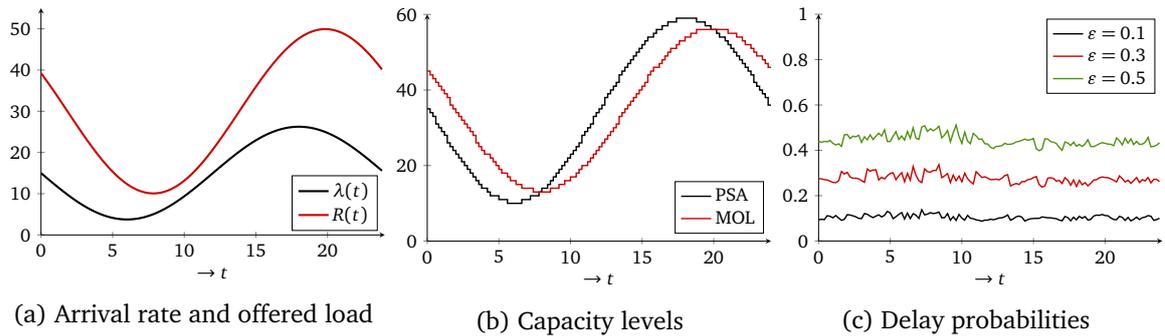

Because time-varying capacity allocation is an issue that recurs in many practical settings, this has been the topic of many works; see e.g.~\cite{Feldman2008,Defraeye2013,Whitt2007,Whitt2013,Liu2012,Liu2014,Liu2017,He2016}. For an accessible overview of queueing-theoretical methods for determining capacity levels under time-varying demand, see Kolesar et al.~\cite{GreenKolesar2007} and references therein. Whitt~\cite{Whitt2017} provides a review of queueing models with time-varying demand.

\section{Convergence rates}\label{sec:refinements}

By now, it is clear that the QED paradigm is based on limit theorems that apply when systems become infinitely large. In practice, even large systems are finite, which makes it important to quantify the error made in approximating a finite system by a limiting object. As it turns out, QED approximations are in many cases highly accurate, already for relatively small or moderately sized systems. In this section we show how to quantify these errors by determining the rate of convergence of certain performance measures to their asymptotic limits.
A first sign of this was seen through the accuracy of the asymptotic dimensioning schemes in Section \ref{sec:dim}.
These convergence rates are typically of order $1/\sqrt{s}$ with $s$ the system size. This again confirms the deep connection with the CLT with a typically error also of order $1/\sqrt{s}$ but then with $s$ the number of random variables in the sum.

\subsection{Bounds}
A convergence rate can also be interpreted as the (main) error made when using the QED limits as approximations for the real performance measures. Whenever we find ways to obtain explicit and precise descriptions of the convergence rates, this can also be used to correct the limiting expression for the finite size of the system. We will also show how such effective corrections can be obtained and applied directly in the QED framework.

Recall that when $\lambda$ is a positive integer, $\Pois(\l)$ can be written as the sum of $\lambda$ $\Pois(1)$ random variables.
A more general version of the CLT in Theorem \ref{thm:CLT} related to the Berry-Ess\'{e}en bound, see e.g.~\cite[Sec.~XVI.5]{feller2}, implies that
\begin{equation}\label{classicclt}
\mathbb{P}(\Pois(\l)\leq \sl)=\Phi(\beta)+O(\lambda^{-1/2}),
\end{equation}
as $\lambda\rightarrow\infty$ with $\sl$ as in \eqref{beta}.
Comparing \eqref{classicclt} with \eqref{eq:first_clt_poisson}, \eqref{classicclt} not only shows convergence of $\P(\Pois(\l)\leq \sl)$ to $\F(\b)$, but also quantifies (roughly) the convergence rate as $O(\l^{-1/2})$.
To obtain better estimates for the error of order $1/\sqrt\l$, one can derive asymptotic expansions. There are various general theorems that yield asymptotic expansions for $\mathbb{P}(A_\lambda\leq s)$ in ascending positive powers of $\lambda^{-1/2}$, see, e.g.~\cite{barndorff,batta,feller2,hwang,johnsonkotzkemp,petrov}. One example would be the Edgeworth expansion, which for the Poisson distribution yields, see \cite[Eq.~(4.18)]{barndorff},
\begin{equation}\label{g234}
\mathbb{P}(\Pois(\l)\leq \sl)=\Phi(\beta)-\frac{\varphi(\beta)(\beta^2-1)}{6\sqrt{
\lambda}}+O(1/\lambda).
\end{equation}

The technical challenge in determining convergence rates is that we need to establish an asymptotic expansion rather than just the limit theorem. We shall demonstrate this for the $M/M/\sl$ queue using the asymptotic evaluation of integrals through the Laplace method.
The formula $C(\sl,\lambda)$ in its basic form is only defined for integer values of $\sl$. An extension of this formula that is well defined for all real $\sl>\lambda$ is given by (see for example Jagers \& Van Doorn (1986))
\begin{equation}
\label{jagersvandoornc}
C(\sl,\lambda)^{-1} = \lambda \int_0^\infty t \e^{-\lambda t} (1+t)^{\sl-1} \dd t.
\end{equation}

We introduce  the following key parameters:
\begin{align}
\alpha &= \sqrt{-2\sl(1-\rho_\l+\ln \rho_l)},\label{alphadef}\\
\beta &= (\sl-\lambda)/\sqrt{\lambda},\\
\gamma &= (\sl-\lambda)/\sqrt{s} = (1-\rho_\l)\sqrt{\sl} = \beta\sqrt{\rho_\l}.
\end{align}
It has been shown in \cite{Janssen2008b} that $\alpha<\beta$. By expanding $\frac 12 \alpha^2$ in powers of $(1-\rho_\l)$, it easily follows that $\gamma<\alpha$, so we have $\gamma<\alpha<\beta$.

\begin{theorem}\label{thmm1}
For $s>\lambda$,
\begin{equation}
\label{simpleupperbound}
C(\sl,\lambda) \leq \left[\rho_l + \gamma \left(\frac{\Phi(\alpha)}{\varphi(\alpha)} + \frac 23 \frac 1{\sqrt{\sl}}\right)\right]^{-1},
\end{equation}
and
\begin{equation}\label{simplelowerbound}
C(\sl,\lambda) \geq \left[\rho_\l + \gamma \left(\frac{\Phi(\alpha)}{\varphi(\alpha)} + \frac 23 \frac 1{\sqrt{\sl}} +  \frac 1{\varphi(\alpha)} \frac 1{12\sl-1} \right)\right]^{-1}.
\end{equation}
\end{theorem}
Notice that the structure of the bounds (\ref{simpleupperbound}) and (\ref{simplelowerbound}) is quite similar to the Halfin-Whitt approximation $C(\sl,\l) \approx g(\b)$.
Indeed, using $\sl = \l + \b\sqrt\l$ with $\b$ fixed and letting $\l\to\infty$, one can see that $\alpha$ and $\gamma$ both converge to $\beta$.
With the above theorem at hand, convergence of $C(\sl,\lambda)$ towards the Halfin-Whitt function $g(\beta)$ is follows, which provides an alternative proof and confirmation of Proposition \ref{prop:HalfinWhitt_delay_probability}.
More importantly, the bounds \eqref{simpleupperbound}--\eqref{simplelowerbound} are sharp in the QED regime for small and moderate-size systems. The difference between the lower and upper bound is only $O(1/\sl)$
In Table 1, we keep $\beta=1$ fixed and vary $\sl$. The load $\lambda$ is chosen such that $\sl=\lambda+\beta\sqrt{\lambda}$. The quality of the bounds is apparent, even for small systems, and certainly compared to the asymptotic approximation $g(1) = 0.22336$.

\begin{table}[h]
\begin{center}
\begin{tabular}{rrccccccc}
$\sl$   & $\lambda$   & $\alpha$  &  \eqref{simplelowerbound} & $C(\sl,\lambda)$    & \eqref{simpleupperbound} & $\frac{\eqref{simpleupperbound}-\eqref{simplelowerbound}}{C(\sl,\lambda)}$ & \eqref{eq:delay_refinement} & $\frac{|\eqref{eq:delay_refinement} - C(\sl,\l)|}{C(\sl,\l)}$  \\ \hline
    1    &   0.382  &   0.830  &   0.36571  &   0.38197  &   0.39437 &   7.504$\cdot$ $10^{-2}$ & 0.45085 & 1.803 $\cdot$ $10^{-1}$ \\
    2    &   1.000  &   0.879     &   0.32678  &  0.33333  &   0.33936  &   3.772$\cdot$ $10^{-2}$ & 0.36395 & 0.918 $\cdot$ $10^{-2}$ \\
    5    &   3.209  &   0.924  &   0.28886  &   0.29097  &   0.29328  &   1.518$\cdot$ $10^{-2}$ &
   0.30185 & 3.739 $\cdot\ 10^{-2}$ \\
   10    &   7.298  &   0.946  &   0.26937  &   0.27030  &   0.27142  &   7.616$\cdot$ $10^{-3}$ & 0.27540 & 1.886 $\cdot\ 10^{-2}$ \\
   20    &   16.000   &  0.962  &   0.25565 &   0.25608  &   0.25663  &   3.818$\cdot$ $10^{-3}$ & 0.25851 & 9.495 $\cdot\ 10^{-3}$ \\
   50    &   43.411  &   0.976  &   0.24361  &  0.24377 &   0.24398  &   1.531$\cdot$ $10^{-3}$ &
 0.24470 & 3.820 $\cdot\ 10^{-3}$  \\
  100    &   90.488   &   0.983  &   0.23761  &   0.23769  &   0.23779  &   7.665$\cdot$ $10^{-4}$ & 0.23814 & 1.916 $\cdot\ 10^{-4}$ \\
  200    &   186.349   &   0.988  &   0.23340  &  0.23344 &   0.23349  &   3.836$\cdot$ $10^{-4}$& 0.23366 & 9.602 $\cdot\ 10^{-4}$\\
  500    &   478.134   &   0.993  &   0.22969  &   0.22970  &   0.22972  &   1.536$\cdot$ $10^{-4}$ & 0.22979 & 3.848 $\cdot\ 10^{-4}$ \\
 1000    &   968.873   &   0.995  &   0.22783  &   0.22783  &   0.22784  &   7.683$\cdot$ $10^{-5}$ & 0.22788 & 1.926 $\cdot\ 10^{-4}$\\
\end{tabular}
\end{center}
\caption{Results for the bounds on $C(\sl,\lambda)$ for $\beta=1$.}
\end{table}

We by now know that $C(\sl,\l)\to g(\b)$ and D'Auria~\cite{DAuria2012} proved that $C(\sl,\l) \geq g(\beta)$ for all $\l,\beta>0$.
Using the bounds in \eqref{simpleupperbound} and \eqref{simplelowerbound}, it was shown by Janssen et al.~\cite{Janssen2011} that as $\lambda\rightarrow\infty$,
\begin{equation}
C(\sl,\l) \approx g(\beta) + g_{\bullet}(\beta) \frac \beta{\sqrt{\lambda}},\label{eq:delay_refinement}
\end{equation}
with
\begin{equation}
g_{\bullet}(\beta) = g(\beta)^2 \left[\frac{1}{3} +  \frac{\beta^2}{6} + \frac{\Phi(\beta)}{\phi(\beta)}\left(\frac{\beta}2 + \frac{\beta^3}{6}\right)\right].
\end{equation}
This result can be interpreted as the counterpart of \eqref{g234}, but then not for the Poisson distribution in the CLT regime, but for the delay probability in the QED regime.
In Table 1 we see that \eqref{eq:delay_refinement} leads to much sharp approximations than the original asymptotic approximation $g(1) = 0.22336$.


\subsection{Optimality gaps}
Given these refinements to the asymptotic delay probability, we revisit the cost minimization problem discussed in Section \ref{sec:dim}, and  ask ourselves what can be said about the associated optimality gaps when dimensioning principles based on the asymptotic approximations are used.

Recall that under linear cost structure, we aim to find the minimizing value $s^*_\l$ of $K(s,\l)$ as in \eqref{eq:optimization_objective} (we omit the argument $r$ in this section for brevity).
Since $K(\sl,\l) \to K_*(\beta)$ as $\l\to\infty$ with $\sl = \l+\beta\sqrt{\l}$,
we alternatively considered asymptotic minimizer $s^{\rm QED}_\l = [\lambda + \beta^*\sqrt{\l}]$ with $\beta^*$ minimizing $K_*(\beta)$, and
Figure \ref{fig:MMs_staffing_levels_optimization} illustrated the accuracy of this asymptotic dimensioning scheme of systems of various sizes.
Indeed, Borst et al.~\cite{Borst2004} showed that $s^{\rm QED}_\l$ is asymptotically optimal in the sense that
\begin{equation}
K\big(s^{\rm QED}_\l,\l\big) = K( s_\l^*, \l) + o(\sqrt{\l}).
\end{equation}
The corrected approximation for the delay probability in \eqref{eq:delay_refinement}, however, provides a means to improve the accuracy of $s^{\rm QED}_\l$.
Namely, by substituting \eqref{eq:delay_refinement} into \eqref{eq:optimization_objective}, it is clear that we can write
\begin{equation}\label{eq:cost_approx}
\frac{K(\sl,\l)}{\sqrt{\l}} \approx K_*(\beta) + \frac{g_\bullet(\beta)}{\sqrt{\l}} =:K_\bullet(\beta) ,
\end{equation}
with an error that is of order $O(1/\l)$ for uniformly bounded $\beta$ and $g_\bullet(\beta) := g_*(\beta)/\beta$.
If we consider the approximated cost function $K_\bullet(\beta)$ in \eqref{eq:cost_approx}, and let $\beta^*_\l$ be the associated minimizer, then we expect the refined square-root rule $s_\l^\bullet := [\l+\beta^*_\l \sqrt{\l}]$ to give a better approximation to the true optimizer $s_\l^*$.
It is shown in Janssen et al.~\cite{Janssen2011}, by invoking Taylor's theorem, that $\beta^*_\l = \beta^* + \beta_\bullet / \sqrt{\l} + O(1/\l)$ with
\begin{equation}
\beta_\bullet = {-}\,\frac{ \beta^* g_\bullet'(\beta^*)}{ K_*''(\beta^*) + 2r }.
\end{equation}
The resulting refined square-root rule
$s^\bullet_\l = [s_\l^{\rm QED} + \beta_\bullet]$
indeed yields an improvement over the original square-root rule in terms of the optimality gap. Namely, see \cite[Thm.~2]{Janssen2011},
\begin{equation}
K(s^\bullet_\l,\l) = K(s^*_\l,\l) + O(1/\sqrt{\l}).
\end{equation}
Observe that the characterization of $s^\bullet_\l$ as an $O(1)$ correction to the original square-root rule \eqref{beta} provides a rigorous mathematical underpinning for the exceptionally good performance of the QED dimensioning scheme observed in Figure \ref{fig:MMs_staffing_levels_optimization}.

In the context of $M/M/s+M$ queues, Zhang et al.~\cite{Zhang2012} obtained similar results on optimality gaps.
Motivated by the results in \cite{Janssen2011,Zhang2012}, Randhawa \cite{Randhawa2014} takes a more abstract approach to quantify optimality gaps of asymptotic optimization problems.
He shows under generally assumptions that when the approximation to the objective function is accurate up to $O(1)$, the prescriptions that are derived from this approximation are $o(1)$-optimal.
The optimality gap thus asymptotically becomes zero.
This general setup is shown in \cite{Randhawa2014} to apply to the $M/M/s$ queues in the QED regime, which confirmed and sharpened the results on the optimality gaps in \cite{Janssen2011,Zhang2012}.
The abstract framework in \cite{Randhawa2014}, however, can only be applied if refined approximations as discussed above are available. Optimality gaps in settings with admission control in the QED regime, based on a trade-off between revenue, costs and service quality, have  been studied  in \cite{Sanders2016}.

\subsection{Refinements}

A downside of heavy-traffic analysis is that the results are of an asymptotic nature, and therefore approximations. Obtaining corrections or refinements is one of the main goals of many research efforts, and the demonstration in the previous two subsections, is only a small part of a richer and active line of research. In the QED context, this leads to the question of whether the three universal properties provide the correct insight if the system size
is only moderate, or if the efficiency hypothesis (\ref{eq:efficiency}) is not exactly satisfied.

In the setting of a single server, Siegmund \cite{siegmund} proposed a corrected diffusion approximation for the waiting time. In heavy traffic, its distribution is approximately exponential. Siegmund gave a precise estimate of the correction term, nowadays a classical result and textbook material, cf. \cite[p.~369]{Asmussen2003}.
Siegmund's first order correction has been extended by Blanchet \& Glynn \cite{Blanchet2006}, who give a full series expansion for the $G/G/1$ waiting time distribution in heavy traffic.
A result similar to \cite{siegmund} has been presented in the QED context for the $M/M/s$ queue in Section \ref{sec:refinements} of this survey.
A common threat of these approaches is that detailed information on the pre-limit distribution needs to be available.

In addition to corrected diffusion approximations, a number of other refinements exist in the literature that provide improved (w.r.t.~the heavy-traffic limit) approximation of the invariant distribution. One class of such approximations is based on variations of Stein's method \cite{Braverman2015,Braverman2016}. Another class of approximations is based on the idea to consider the diffusion limit of a Markovian queue, and to replace the drift and diffusion coefficients by terms that depend on the parameters in the pre-limit. The goal is to improve the convergence rate in the QED regime and to make the approximations accurate in other scaling regimes, hence the term {\em universal approximations}. We refer to \cite{Gurvich2014,Gurvich2014a,Huang2016} for a more in-depth discussion, and explain the idea of modifying a diffusion in the context of the Halfin-Whitt diffusion, following an idea of Dai \& Braverman~\cite{DaiBraverman}.

Recall from Theorem \ref{thm:Halfin_Whitt_diffusion} that the scaled queue length process in the QED regime converges to a diffusion process with infinitesimal drift $m(x) = -\beta -x \mathbbm{1}_{\{x\leq 0\}}$ and infinitesimal variance $\sigma^2(x)=2$.
$\beta$ can be expressed in terms of the pre-limit characteristics by the expression $\beta = (s-\lambda)/\sqrt{\lambda}$.
The idea in \cite{DaiBraverman} is now to replace the diffusion coefficient and consider \begin{equation}
\sigma_\lambda^2(x) = 1+ \mathbbm{1}_{\{x> -\sqrt{\lambda}\}} \left(1-\frac{m(x)}{\sqrt{\lambda}}\right).
\end{equation}
The resulting approximation for the steady-state density is explicit and it is shown in \cite{DaiBraverman} that the resulting distributional approximation has an error of the order $1/\lambda$, while the QED approximation has a much larger error of order $1/\sqrt{\lambda}$. Though the associated approximation for the delay probability is worse than the approximations and bounds presented in Section \ref{sec:refinements} of this paper, the idea of modifying the limiting diffusion appropriately seems to be of high potential, and worthy of futher investigation. The same can be said about Stein's method. Another line of research we think deserves
attention is to develop non-asymptotic bounds that are accurate in a QED setting. Recent work in this direction is \cite{GoldbergLi2017}.

\section{Extensions}\label{extsec}

By now we should have developed a good understanding for why the mathematical theory that comes with the QED regime for many-server systems ranks among the most celebrated principles in applied probability.
The goal of the present section is to provide a survey of results for models that are more elaborate than the basic models discussed so far. In particular, we shall consider more elaborate models that incorporate various forms of user behavior (such as abandonments and strategic behavior), and consider the impact of blocking in systems with finite waiting rooms, as well as loss networks. We also consider parameter uncertainty, systems with load balancing, and non-Markov G/G/1 systems. 
This list of extensions is far from exhaustive, but gives the reader a taste of how the fundamental QED principles given in Section 3 carry over of should be adapted in more elaborate settings.  
For multi-class job types we refer to~\cite{Harrison2004,Armony2004,Atar2014,Gurvich2008,Gurvich2009,Tezcan2010,maglaraszeevi,Atar2012,Atar2016}. Extensions to heterogeneous servers can be found in~\cite{Armony2005,Armony2010,Mandelbaum2012b,Stolyar2010},  and congestion control mechanisms are considered in ~\cite{Sanders2014,Kocaga2015,Atar2005,Atar2016,Atar2012,Borgs2014,Gurvich2009,Janssen2013}.

\subsection{Abandonments}

So far, we have surveyed standard systems in which
all arriving jobs join the queue and stay until eventually being processed by one of the servers.
One model extension that is featured prominently in the literature is abandonment caused by customer impatience, in which case customers leave the system without being served \cite{Gans2003,Brown2005,Palm1957}.

The canonical model for abandonments is the $M/M/s+M$ or Erlang A model \cite{Palm1957,Garnett2002}, with similar dynamics as the $M/M/s$ queue, with the additional feature that each job is assigned an i.i.d.~patience time, which is exponentially distributed with mean $1/\theta$.
If a job's patience time expires before reaching an available server, the job leaves (abandons) the system.
As the number of jobs in the Erlang A queue remains a birth-death process, its stationary distribution and associated performance measures are fairly well-understood, also in the QED regime \cite{Garnett2002,Zeltyn2005,Zhang2013}.
Garnett et al.~\cite{Garnett2002} and Zeltyn \& Mandelbaum~\cite{Zeltyn2005} showed that in the QED regime, with $\sl = \l + \beta\sqrt{\l}$ and $\l\to\infty$,
\begin{align}
\P({\rm delay}) &\to \left( 1 + \sqrt{\theta}\,\frac{k(\beta/\sqrt{\theta})}{k({-}\beta)}\right)^{-1}
\end{align}
and
\begin{align}
\sqrt{\l}\,\P({\rm abandon}) &\to
\frac{ \sqrt{\theta}\,k(\beta/\sqrt{\theta})- \beta }
{ 1 + \sqrt{\theta}\,k(\beta/\sqrt{\theta})/k({-}\beta)},
\end{align}
where $k(\beta) = \f(\beta)/\F({-}\beta)$.
Hence, the universal QED properties, discussed in Section \ref{sec:properties}, remain intact when the model includes abandonments.
Moreover, the probability that a job abandons vanishes at rate $O(1/\sqrt{\l})$ as $\l\to\infty$.
In \cite{Zeltyn2005}, the stationary QED limits for more generally distributed patience time were derived, for which similar limiting behavior is proved.
More surprisingly, it is shown that the limit is insensitive to the patience
time distribution as long as its density at 0, i.e. the probability of abandoning immediately upon arrival, is fixed.
On the process level, the appropriately scaled queue length process of the $M/M/\sl+M$ model in the QED regime can be shown to converge to a piecewise-linear Ornstein--Uhlenbeck process with drift terms
\begin{equation}
m(x) = \left\{
\begin{array}{ll}
-\beta-\theta x, & \text{if }x> 0,\\
-\beta-x, & \text{if } x \leq 0
\end{array}\right.
\label{2222}
\end{equation}
and infinitesimal variance $\sigma^2(x) = 2$, see e.g.~\cite{Garnett2002}.
Notice that for $\theta = 0$, we retrieve the Halfin-Whitt diffusion in Theorem \ref{thm:Halfin_Whitt_diffusion}.
Under more general assumptions, \cite{Mandelbaum2012a} characterizes the QED limiting process for the $G/GI/s+GI$ queue.
More specifically, they find that the QED limit of the $G/M/s+GI$ queue is still a piecewise-linear Ornstein-Uhlenbeck process.

The general $G/G/s+G$ queue under various modeling assumptions and its limiting process in the QED regime has been studied in \cite{Garnett2002,Gans2003,Whitt2006,Mandelbaum2009,Zeltyn2005,Mandelbaum2012a,Kang2012,Dai2010,Reed2012,Jennings2012,Zhang2013}.
These works also include the case where the system is balanced from the point of view of the abandonment probability, which related to the efficiency driven regime in our setting.
Surveys on systems with abandonments are Ward \cite{Ward2012} and Dai \& He \cite{Dai2012}.

\subsection{Finite waiting space}
We have assumed so far that systems have infinite buffers for storing delayed jobs.
Systems in applications such as data centers and hospitals, however, are often limited in the number of jobs that can be held simultaneously.
Depending on the practical setting and admission policy, if the maximum capacity, say $k$, is reached, newly arriving jobs can either leave the system immediately (blocking), reattempt  getting access later (retrials) or queue outside the facility (holding).
In any case, expectations are that the queueing dynamics within the resource sharing facility are affected considerably in the presence of such additional capacity constraints.

We illustrate these implications through the $M/M/s/k$ queue, that is, the standard $M/M/s$ queue with additional property that a job that finds upon arrival $k$ jobs already present in the system is blocked/lost.
To avoid trivialities, let $k\geq s$.
Since the mean workload reaching the servers is less than in an finite buffer $(k=\infty)$ scenario, one expects less congestion and resource utilization.

Consider the $M/M/\sl/k_\l$ in the QED regime.
So, let $\l$ increase while $\sl$ scales as in \eqref{beta}.
We then ask how $k_\l$ should scale along with $\l$ and $\sl$ to maintain the non-degenerate behavior as seen in Section \ref{ss:ex1}.
We provide a heuristic answer.
Let $Q^{(\sl,k_\l)}$ and $W^{(\sl,k_\l)}$ denote the number of jobs in the system and delay in the $M/M/\sl/k_\l$ queue in steady state.
If there were no finite-size constraints, then through \eqref{eq:diff_1}--\eqref{eq:diff:3}, we find as $\l\to\infty$
\begin{equation}
\P( Q^{(\sl)} \geq k_\l ) = \P\left( \frac{Q^{(\sl)}-\sl}{\sqrt\sl} \geq \frac{k_\l-\sl}{\sqrt\sl}\right) \to g(\b)\, \ee^{-\b \gamma},
\end{equation}
where $\gamma = \lim_{\l\to\infty} (k_\l-\sl)/\sqrt{\sl}$.
Hence, asymptotically the finite-size effects only play a role if the extra variability hedge of $k_\l$ is of order $\sqrt{\sl}$ (or equivalently $o(\sqrt{\l})$).
Furthermore, if the variability hedge is $o(\sqrt{\l})$, then we argue that asymptotically, all jobs that do enter the system have probability of delay equal to zero.
More formally, under the \textit{two-fold scaling rule}
\begin{equation}
\label{eq:intro_twofold_scaling_rule}
\left\{
\begin{array}{ll}
\sl = \l + \beta\sqrt{\l} + o(\sqrt{\l}),\\
k_\l = \sl + \gamma \sqrt{\sl} + o(\sqrt{\l}),
\end{array}
\right.
\end{equation}
it is not difficult to deduce that, see e.g.~\cite{masseywallace},
\begin{equation}
\P({\rm delay}) \to \left( 1 + \frac{\beta\,\F(\beta)}{(1-\ee^{-\beta\gamma})\f(\beta)}\right)^{-1}, \quad \text{as } \l\to\infty,
\end{equation}
which is strictly smaller than $g(\beta)$ in \eqref{fig:delay_probs_HW_MMs}, but still bounded away from both 0 and 1 and thus the balance property holds.
Furthermore, the buffer size of the queue is $n_\l-\sl = \gamma\sqrt{\sl}$, so that by Little's law, the mean delay of an admitted job is $O(1/\sqrt{\sl})$, implying that the QoS property holds.
Even though resource utilization in the $M/M/\sl/n_\l$ is less efficient than in the queue with unlimited waiting space, it can be shown that $\rho_\l\to 1$ as $\l\to\infty$.
Hence, all three key characteristics of the QED regime are carried over to the finite-size setting if one uses \eqref{eq:intro_twofold_scaling_rule}.

On a process level, adding a capacity constraint translates to adding a reflection barrier to the normalized queue length process $\bar{Q}^{(\sl,n_\l)} = (Q^{(\sl,n_\l)} -\sl ) /\sqrt{\sl}$, at $\gamma$, as is illustrated by the sample paths of $\bar{Q}^{\sl,n_\l}$ for three values of $\l$ in Figure \ref{fig:sample_paths_MMsn}.
\begin{figure}
\centering
\begin{subfigure}{0.32\textwidth}
\centering
\includegraphics[scale=0.8]{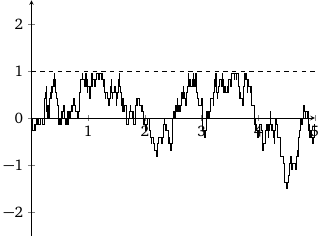}
\caption{$\l = 50$}
\end{subfigure}
\begin{subfigure}{0.32\textwidth}\centering
\includegraphics[scale=0.8]{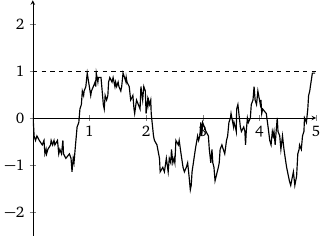}
\caption{$\l=100$}
\end{subfigure}
\begin{subfigure}{0.32\textwidth}\centering
\includegraphics[scale=0.8]{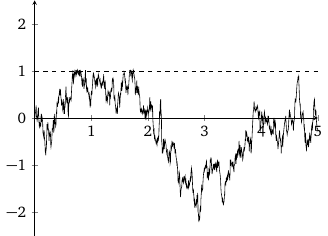}
\caption{$\l=500$}
\end{subfigure}
\caption{Sample paths of the normalized queue length process $\bar{Q}^{(\sl,n_\l)}(t)$ with $\l = 50,\, 100$ and $500$ under scaling \eqref{eq:intro_twofold_scaling_rule} with $\beta=0.5$ and $\gamma = 1$.}
\label{fig:sample_paths_MMsn}
\end{figure}
Indeed, non-degenerate limiting behavior can be expected when the additional space $\g\sqrt{\sl}$ is of the same order as the natural fluctuations of the arrival process; see~\cite{masseywallace}.

\subsection{Strategic behavior}

The purpose of this section is to show that the universal QED properties may no longer hold when there is strategic interaction between the system operator and potential users.
In several applications users have the option to join a certain congestion dependent service or not, leading to a game theoretic setting where the provider of a service maximizes profit, and users decide to join a service depending on their utility, possibly involving the mean delay. If the market size is large, the QED capacity allocation rule can emerge endogenously, though it is possible to obtain other scaling rules as well.
Examples of such studies include \cite{KumarRandhawa2010,WardWierman2016,NairWiermanZwart2016}. For illustrative purposes, we briefly describe the model and results of Nair et al.~\cite{NairWiermanZwart2016} in more detail.

A user needs to decide whether or not to use a congestion-dependent service which is free for the user (and supported by advertisements, think of Google or Facebook). If the total user base that uses the service has magnitude $\lambda$, the user receives a utility $V(\lambda)$ (this may be increasing with $\lambda$ in a social network context), and a congestion-dependent dis-utility $\xi(s,\l)$, chosen according to the mean delay in the $M/M/s$ queue, i.e. $\xi(s,\l) = C(s,\l)/(s-\lambda)$ for $\lambda<s$ and $\infty$ otherwise. Given the choice of a number of data processing units of the service provider, an infinitesimal user will join if and only if $V(\lambda) - \xi(s,\l)$ is non-negative. The total market size of the user base is equal to $\Lambda$, which is assumed to be large. For illustrative purposes, we restrict to the case where the entire user population can cooperate and therefore the total arrival rate becomes
\begin{equation}
\hat \lambda_\Lambda(s) = \max \Big\{ \arg\max_{\lambda\in [0,\Lambda]} [ \lambda V(\lambda) - \lambda \xi(s,\l)]  \Big\}
\end{equation}

The firm optimizes its revenue given this user behavior. The cost of each resource is scaled to $1$, and the average advertisement revenue per unit of users is set to $b_1$. In this case the optimal number of services $k^*(\Lambda)$ becomes
 \begin{equation}
k_\Lambda^* = \max \Big\{ \arg\max_{k\geq 0} [b_1 \hat \lambda_\Lambda (k) - k] \Big\}.
\end{equation}
It is possible to determine how $k_\Lambda^*$ scales with $\Lambda$. As is shown in Theorem 1 of \cite{NairWiermanZwart2016}, if $\alpha = \lim_{\lambda\rightarrow\infty} U'(\lambda) \in (0,\infty)$ (which is the case if $V$ is converging to a constant, corresponding with an online service like Google), then there exists a strictly positive and decreasing function $\beta$ of $\alpha$ such that $k_\Lambda^* = \Lambda + \sqrt{\beta (\alpha)\Lambda}\,(1+o(1))$.
In the case $V(\lambda) = \lambda^v$ for some $v>0$, then $\alpha=\infty$ and users are more interested to join a service if other users are present (as is the case in a social network like Facebook). In this case, the firm can give less QoS: the number of spare servers becomes of the order $\sqrt{\Lambda^{1-c}}$.
If users cannot collaborate, the firm only needs two spare servers to maximize its profit: the choice $k^*_\Lambda = \Lambda+2$ makes the entire user population join the network. This is an example of what is called a {\em tragedy of the commons}. There are many additional opportunities for research in this domain; the recent monograph \cite{Hassinnewbook} on the interface of game theory and queueing provides an excellent starting point.
{Another interesting research line is the analysis of the effect of delay announcements on consumer behavior in a many server setting, see \cite{Armony2009, Huang2017}.
 A recent survey on this topic is \cite{Ibrahim2018}.}

\subsection{Networks}
The models shown so far all are all single-station models. The analysis of networks in the QED regime is more challenging; {see e.g.~\cite{mandelbaum1998strong} for a QED analysis of Markovian networks with time-varying rates and time-varying number of servers, using the technique of strong approximations to obtain functional functional central limit theorems. Parallel service systems with multiple
service pools and multiple customer classes can also be viewed as network extensions of single stations; see 
\cite{Atar2005} for a diffusion control problem and 
\cite{Gurvich2009} for a strong approximation approach. In both papers the QED regime plays a key role. 
In this section, we restrict to explaining how the fundamental QED properties can be extended to a tractable class of loss networks.}

A {\em loss network} is an extension of the Erlang B model, and is especially relevant for the analysis of communication networks.
Consider a network with $J$ links, and suppose that link $j$, $j=1,\ldots,J$, comprises $C_j$ circuits (servers). There are $R$ classes of calls called routes. A call on route $r$ uses $A_{jr}$ circuits from link $j$, where we take $A_{jr}$ to be either 0 or 1. Calls of route $r$ arrive according to a Poisson arrival process of rate $\lambda_r$ and a call is blocked if the appropriate servers are not available.
Assuming unit exponential services on each route, it can be shown that the invariant distribution $\pi (n)$ can be written as a ratio of two Poisson probabilities. Specifically, let $N$ be an $R$-dimensional vector of independent Poisson random variables where the rate of $N_r$, $r\in R$, equals $\lambda_r$. 
Now \begin{equation}
\pi(n) = \frac{P(N = n) }{\P(AN \leq C) } I(An \leq C),
\end{equation}
with $C=(C_1,...,C_J)$
Unfortunately the computation of the normalizing constant $\P(AN \leq C)$ is nontrivial for large systems.
It is possible to develop a Gaussian approximation using a central limit approach which can be seen as an extension of our efficiency hypothesis (\ref{eq:efficiency}).
To have all links in the network critically loaded, one considers the case where $\lambda_r = \lambda \nu_r$ and replaces $C$ with $\lambda C + \beta \sqrt{\lambda}$, with $\lambda$ a scaling parameter, as before and
$\beta$ a vector. It is possible to show that all links in the network are critically loaded in this case if $A\nu=C$. For cases where only a subset of links in the network is critically loaded, one must proceed in a much
more delicate manner; see  \cite{HuntKelly1989}.

%

The normalizing constant can, under our scaling hypothesis, be written as $\P((AN -\lambda C) \leq \sqrt{\lambda}\beta )$, which converges to a multivariate Gaussian distribution, as $A\lambda \nu=\lambda C$.
The analogue of (\ref{eq:QoS}) can be seen as the fact that each user has
a dedicated group of service once it is admitted, and the probability of blocking decays at a rate $O(1/\sqrt{\lambda})$, which is analogous to the blocking rate in a simple Erlang B system, relating to (\ref{eq:balance}).
For more details on these properties and more background we refer to \cite{Kelly1991}, which is still a valuable source of information, and the more recent \cite{KellyYudinova2014}. For recent progress on computational procedures we refer to \cite{IBMpaperSigmetrics,IBMpaperquesta}.

Other network extensions of QED principles have been established, though it is typically hard to derive explicit results for the associated limiting distributions and/or processes. For work on fork-join networks in the QED regime, we refer to \cite{Lu2015,Pang2016,Lu2017}, while bandwidth sharing networks in the QED regime have been investigated in \cite{ReedZwart2014}.
For network analogues of the two-fold scaling rule presented in Section 6.2 (in particular Equation  \eqref{eq:intro_twofold_scaling_rule}), see \cite{Khudyakov2006,YomTov2010,Tan2012}.

\subsection{Parameter uncertainty} 

Models describing multi-server systems typically assume perfect knowledge on the model primitives, including the mean demand per time period.
For large-scale  systems, the dominant assumption in the literature is that demand arrives according to a non-homogeneous Poisson process, just as in Section \ref{sec:dim}, which translates to the assumption that arrival rates are known for each basic time period (second, hour or day).
In practice, however, estimates for mean demand typically rely on historical data, and are therefore subject to uncertainty.
This parameter uncertainty is likely to affect the effectiveness of capacity sizing rules.
{
Examples of studies in staffing or resource allocation rules under parameter uncertainty are \cite{maman,koolejongbloed,Gurvich2010,Bassamboo2010,Whitt2006}; see also the references therein.}

As an illustration, consider a resource allocation problem with Poisson $\lambda$ arrivals and exponential ($\mu$) servers. Suppose that $\mu=1$, and $\lambda$ is unknown.
For instructive purposes, we make a resource allocation decision $s$ based on the infinite server approximation $\P(\Pois(\l) > \sl)\leq \varepsilon$. In case $\lambda$ is known and large, the choice $\sl=\lambda + \beta \sqrt{\lambda}$, with $\beta = 1-\Phi^{-1}(\varepsilon)$, would be natural, see \eqref{appis}.
If $\lambda$ is not known, but needs to be estimated from data, it is instructive to see how the choice of $s$ is affected. Suppose we have an estimator $\hat\lambda$ of $\lambda$ which is approximate normally distributed with standard deviation $\sigma$. When would it be appropriate to simply take $s=\hat\lambda + \beta \sqrt{\hat\lambda}$?
To obtain some insight, we use the approximation $\Pois(\l) \sim \lambda +G\sqrt{\lambda}$ and assume $\hat{\l} = \l + G_0\sigma$, where $G$ and $G_0$ are independent standard normal variables.
Then we see the following: if $\lambda$ is large, we need to pick $s$ such that $P(\hat\lambda + \sigma G_0+ G\sqrt{\hat \lambda} > s) = \varepsilon$, yielding $s = \hat \lambda + \beta \sqrt{ \sigma^2 + \hat \lambda}$. If $\sigma^2$ is of the order $\hat \lambda$, it follows that the naive rule $s=\hat \lambda + \beta\sqrt{ \hat \lambda}$ leads to poor system performance.
It would be valuable to develop similar quantitative insights for more realistic models.

A related yet fundamental difficulty arises when fluctuations in demand are larger than anticipated by the Poisson assumption. In this case, the Poisson assumption is wrong, and standard
QED rules need to be modified, even when there is an infinite amount of data.
Indeed, although natural and convenient from a mathematical viewpoint, the Poisson assumption often fails to be confirmed in practice.
A deterministic arrival rate implies that the demand over any given period is a Poisson random variable, whose variance equals its expectation.
A growing number of empirical studies of service systems shows that the variance of demand typically exceeds the mean significantly, see \cite{Avramidis:2004,Bassamboo2010,Bassamboo2009,Brown2005,Chen2001,Gans2003,Gurvich2010,koolejongbloed,kimwhitt,maman,Mehrotra2010,Robbins2010,Steckley2009,Zan2012}. The feature that variability is higher than one expects from the Poisson assumption is referred to as \textit{overdispersion}.

Due to its inherent connection with the CLT, the square-root rule relies heavily on the premise that the variance of the number of jobs entering the system over a period of time is of the same order as the mean.
Subsequently, when stochastic models do not take into account overdispersion, resulting performance estimates are likely to be overoptimistic. The system then ends up being under-provisioned, which possibly causes severe performance problems, particularly in critical loading.
To deal with overdispersion, existing capacity sizing rules like the square-root rule need to be modified in order to incorporate a correct hedge against (increased) variability.
Following our findings in Section \ref{sec:properties}, the following adapted capacity allocation rule may be proposed
\begin{equation}
s = \mu_A + \beta \sigma_A,
\end{equation}
where $\mu_A$ and $\sigma_A$ are the mean and standard deviation of demand per period, respecively, and $\beta>0$.
This is similar to \eqref{beta} in which the original variability hedge is replaced by an amount that is proportional to the square-root of the variance of the arrival process.
In \cite{Mathijsen2016}, it is shown that this rule indeed leads to QED-type behavior in bulk-service queues as the system size grows.
For the $M/M/s$ queue, this has been studied in \cite{maman} and \cite{Bassamboo2010}, but more work in this area seems necessary.

\subsection{Load balancing}

The analysis and design of load balancing schemes has attracted strong renewed interest in the last several years, mainly motivated by significant challenges involved in assigning tasks (e.g.~file transfers, compute jobs, data base look-ups) to servers in large-scale data centers. Load balancing schemes provide an effective mechanism for improving QoS experienced by users while achieving high resource utilization levels, goals that are perfectly aligned with the QED regime. A distinguishing feature of such systems, however, is that there is no centralized queue, so that an incoming job should be forwarded instantaneously from the dispatcher to one of the servers. To achieve QED optimality, communication is needed between the dispatcher and servers.
This can cause prohibitive communication burden in large-scale deployments, and asks for assessing load balancing schemes in terms of trade-offs between  performance and implementation overhead.

A naive example of a load balancing scheme is Round Robin, a cyclic scheme that requires no communication, under  which
every $s$-th job is assigned to the same server. For Poisson arrivals and service requirements equal to a constant, Round Robin achieves `perfect load balancing' among servers and the delay distribution is the same as that of a single server serving every $s$th arrival of a Poisson input, or rather, Erlang input. In that case the delay distribution can be approximated by a Gaussian random walk, and all three structural properties are still justified. If deterministic job sizes are being replaced with general job sizes, the system still operates in heavy traffic, and the probability of delay converges to a value in the interval $(0,1)$, but the mean delay will no longer be of the order $O(1/\sqrt{\lambda})$ but constant, so that the third structural QED property no longer holds.

A  more involved example concerns the Join-the-Shortest-Queue (JSQ) scheme and several of its variations, such as versions where the shortest of $d=d(s)$ randomly chosen queues is selected \cite{Mitzenmacher01,VDK96,BLP12,LM06, LN05, G05, BL12, EG16,GTZ16}. In recent years several new results were discovered for JSQ($d(s)$) multi-server systems that operate in the QED regime $(s - \lambda(s)) / \sqrt{s} \to \beta > 0$ as $s \to \infty$.
Eschenfeldt and Gamarnik~\cite{EG15} considered the JSQ scheme with
 $d(s)=s$ and introduced a properly centered and scaled version of the system occupancy processes. They showed that, as $s\to \infty$, the sequence of processes converges weakly to a system of coupled stochastic differential equations.
%
Although this scaling limit differs from the diffusion limit obtained for the fully pooled $M/M/s$ queue, it shares similar favorable QED properties such as vanishing delay. The downside, however, is that a nominal implementation of JSQ comes with a large communication overhead.
It was recently shown  that for $d(s)$ such that $d(s)/(\sqrt{s} \log(s))\to\infty$ as $s\to\infty$ the diffusion limit of JSQ($d(s)$) corresponds to that for the JSQ policy \cite{MBLW16-3,MBLW2016-5}. This indicates that the overhead of the JSQ policy can `almost' be reduced to O($\sqrt{s}\log s$) while retaining diffusion-level optimality.
Many exciting problems in this area, which is still in its infancy, are still open; particular examples are scaling laws (in $s$) of the amount of memory used by the dispatcher, and the amount of communication overhead per packet. See \cite{van2017scalable} for a dedicated survey on this topic.

\subsection{General interarrival and service times}\label{sub:general}
Now consider the $G/G/s$ queue, the natural extension of the $M/M/s$ queue to generally distributed interarrival times and service times. Establishing QED limits for the $G/G/s$ queue has led to a remarkable research effort of which the majority took place over the last decade. When one moves beyond the exponential and deterministic assumptions, establishing QED behavior becomes mathematically more challenging and most of the analysis has the $G/G/s$ queue in the QED regime has evolved around the characterization of the stochastic-process limit of the centered and scaled process, under various assumptions on the model primitives. We restrict our discussion to developments on the basic $G/G/s$ queue; a more extensive discussion, including work on abandonments, can be found in the surveys \cite{Pang2007,Dai2012}. Puhalskii \& Reiman \cite{Puhalskii2000} analyzed the multi-class queue with phase-type service times in the QED regime. Heavy-traffic limits for queues in which service time distributions are lattice-based and/or have finite support were studied by Mandelbaum \& Mom\v{c}ilovi\'c \cite{Mandelbaum2008} and Gamarnik \& Mom\v{c}ilovi\'c  \cite{Gamarnik2008}. The most general class of distributions was considered by Reed \cite{Reed2009} and Puhalskii \& Reed \cite{Puhalskii2010}, who imposed no assumptions on the service time distribution except for the existence of the first moment.
Both of these papers focus on the queue length process. The paper by Reed~\cite{Reed2009} utilized an ingenious connection with the infinite server queue, a connection which was developed further by Puhalskii \& Reed~\cite{Puhalskii2010} where results from modern empirical process theory (including the usage of outer measures to avoid measurability problems) are used in their full potential. Equally important steps forward concern the usage of measure-valued processes by Kang, Kaspi \& Ramanan \cite{Kaspi2011,Kang2012,Kaspi2013}. Assuming minor additional regularity conditions on the service-time distribution (like a bounded density, and sufficiently many finite moments) the paper \cite{Kaspi2013}  unraveled the structure of the limit process that appears after scaling. A key insight from these works is that the limiting queue length process can be interpreted as a one-dimensional diffusion with a drift that depends on the entire history of the process. This as opposed to the Halfin-Whitt diffusion that comes with exponential service times, where the drift depends on the current scaled queue length only.
As a result, even after taking the limit, the resulting limit process for the $G/G/s$ queue has still a complicated steady-state distribution and it is therefore not surprising that considerably less is known for the corresponding steady-state distribution of the $G/G/s$ queue in the QED regime. An exception is the work by Gamarnik \& Goldberg \cite{Goldberg,Gamarnik2013a}, who performed their analysis under the mild assumption that the service time distribution has finite $(2+\varepsilon)$ moments and revealed suitable analogues of all three structural properties mentioned at the beginning of this section, and in addition, explicit tail bounds for the distribution of the delay have been developed.
{Without aiming to be exhaustive, we also point out the recent preprints focusing on  infinite
second moments \cite{GoldbergLi2017ht} and mean waiting times \cite{GoldbergLi2017}. This tutorial is not focused on heuristic approximations
(for example, based on infinite-server models). A recent paper in this direction is \cite{LiuWhitt2016}.}

Finally, we note that the impact of heavy tails is somewhat different than one might expect: the assumption of a service-time distribution with infinite variance is mainly made for technical purposes
(as seen by the generality of the framework in \cite{Reed2009, Puhalskii2010}), while an extension to interarrival times with infinite variance is changing the nature of the scaling procedure itself. For example, to achieve
the balance property (\ref{eq:balance}) in the $G/M/s$ queue where interarrival times have a power law tail with index $\alpha \in (1,2)$, the proper dimensioning rule is
$s_\lambda = \lambda + \beta \lambda^{(\alpha-1)^{-1}}$, see \cite{Pang2007, PangWhitt2009, ReedZwart2011}.

\section*{Acknowledgments}
The authors are grateful to Gordon Pang and two referees for valuable feedback. The research of BM was supported by NWO Free Competition grant no. 613.001.213. The research of JvL is supported NWO Gravitation Networks grant no. 024.002.003. The research of BZ is supported by NWO VICI grant no. 639.033.413.

\bibliography{bibliography,bibliography stochastic service systems,bibliography load balancing}

\end{document}